\documentclass[a4paper, reqno]{amsart}
\usepackage{amssymb, amsmath, amsthm, bm,  enumerate, marginnote,  tikz, tikz-3dplot}

\usetikzlibrary{arrows.meta, cd}

\usepackage[thinlines]{easytable}

\numberwithin{equation}{section}

\newtheorem{theorem}{Theorem}[section]

\newtheorem{lemma}[theorem]{Lemma}
\newtheorem{proposition}[theorem]{Proposition}

\newtheorem*{hypothesis}{Radial Hypothesis}

\theoremstyle{definition}

\newtheorem{remark}[theorem]{Remark}

\newtheorem*{acknowledgement}{Acknowledgement}

\newcommand{\C}{\mathbb{C}}
\newcommand{\R}{\mathbb{R}}
\newcommand{\Z}{\mathbb{Z}}
\newcommand{\N}{\mathbb{N}}

\newcommand{\ud}{\mathrm{d}}

\begin{document}

\title[Variable coefficient Wolff-type inequalities]{Variable coefficient Wolff-type inequalities and sharp local smoothing estimates for wave equations on manifolds}

\author[D. Beltran, J. Hickman and C. D. Sogge]{David Beltran, Jonathan Hickman and Christopher D. Sogge}

\address{BCAM - Basque Center for Applied Mathematics, Alameda de Mazarredo 14, 48009 Bilbao, Spain.}
\email{dbeltran@bcamath.org}

\address{Eckhart Hall Room 414, Department of Mathematics, University of Chicago, 5734 S. University Avenue, Chicago, IL 60637, USA.}
\curraddr{Mathematical Institute, North Haugh, St Andrews, Fife, KY16 9SS.}

\email{jeh25@st-andrews.ac.uk}

\address{Department of Mathematics, Johns Hopkins University, Baltimore, MD 20193, USA.}
\email{sogge@jhu.edu}


% 	35L05   	Wave equation
%35S30		Fourier integral operators
% 58J40   	Pseudodifferential and Fourier integral operators on manifolds
%58J45   	Hyperbolic equations

\thanks{ The first author was supported by the ERC Grant 307617,  the ERCEA Advanced Grant 2014 669689 - HADE, the MINECO project MTM2014-53850-P, the Basque Government project IT-641-13, the Basque Government through the BERC 2014-2017 program, by Spanish Ministry of Economy and Competitiveness MINECO: BCAM Severo Ochoa excellence accreditation SEV-2013-0323, and an IMA small grant. This material is based upon work supported by the National Science Foundation under Grant No. DMS-1440140 while the second author was in residence at the Mathematical Sciences Research Institute in Berkeley, California, during the Spring 2017 semester. The third author was supported by NSF Grant DMS 1665373.}

\begin{abstract} The sharp Wolff-type decoupling estimates of Bourgain--Demeter are extended to the variable coefficient setting. These results are applied to obtain new sharp local smoothing estimates for wave equations on compact Riemannian manifolds, away from the endpoint regularity exponent. More generally, local smoothing estimates are established for a natural class of Fourier integral operators; at this level of generality the results are sharp in odd dimensions, both in terms of the regularity exponent and the Lebesgue exponent. 
\end{abstract}

\maketitle

%%%%%%%%%%%%%%%%%%%%%%%%%%%%%%%%%%%%%%%%%%%%%%%%%%%%%%%%%%%%%%%%%%%%%%%%%%%%%%%%%%%%%%%%%%%%%%%%

%                                INTRODUCTION

%%%%%%%%%%%%%%%%%%%%%%%%%%%%%%%%%%%%%%%%%%%%%%%%%%%%%%%%%%%%%%%%%%%%%%%%%%%%%%%%%%%%%%%%%%%%%%%%

\section{Introduction and statement of results}

%%%%%%%%%%%%%%%%%%%%%%%%%%%%%%%%%%%%%%%%%%%%%%%%%%%%%%%%%%%%%%%%%%%%%%%%%%%%%%%%%%%%%%%%%%%%%%%%

%                                LOCAL SMOOTHING ESTIMATES

%%%%%%%%%%%%%%%%%%%%%%%%%%%%%%%%%%%%%%%%%%%%%%%%%%%%%%%%%%%%%%%%%%%%%%%%%%%%%%%%%%%%%%%%%%%%%%%%

\subsection{Local smoothing estimates}

Let $n \geq 2$ and $(M, g)$ be a smooth,\footnote{In view of the methods of the present article it is convenient to work in the $C^{\infty}$ category, but the forthcoming definitions and questions certainly make sense at lower levels of regularity.} compact $n$-dimensional Riemannian manifold with associated Laplace--Beltrami operator $\Delta_g$. Given initial data $f_0, f_1 \colon M \to \C$ belonging to some \emph{a priori} class, consider the Cauchy problem
\begin{equation}\label{wave equation}
\left\{\begin{array}{l}
(\partial_t^2 - \Delta_g) u = 0 \\[5pt]
u(\,\cdot\,,0) = f_0, \qquad \partial_t u (\,\cdot\,,0) = f_1.
\end{array}\right. 
\end{equation}
It was shown, \emph{inter alia}, in \cite[Theorem 4.1]{Seeger1991} that for each fixed time $t$ and $1 < p < \infty$ the solution $u$ satisfies\footnote{Given a (possibly empty) list of objects $L$, for real numbers $A_{s,p}, B_{s,p} \geq 0$ depending on some Lebesgue exponent $p$ and/or regularity exponent $s$ the notation $A_{s,p} \lesssim_L B_{s,p}$ or $B_{s,p} \gtrsim_L A_{s,p}$ signifies that $A_{s,p} \leq CB_{s,p}$ for some constant $C = C_{L,n,p,s} \geq 0$ depending only on the objects in the list, $n$, $p$ and $s$. In such cases it will also be useful to sometimes write $A_{s,p} = O_L(B_{s,p})$. In addition, $A_{s,p} \sim_L B_{s,p}$ is used to signify that $A_{s,p} \lesssim_L B_{s,p}$ and $A_{s,p} \gtrsim_L B_{s,p}$.}
\begin{equation}\label{fixed time estimate}
\|u(\,\cdot\, , t)\|_{L^p_{s - \bar{s}_p}(M)} \lesssim_{M,g} \|f_0\|_{L^p_{s}(M)} + \|f_1\|_{L^p_{s-1}(M)}
\end{equation}
for all $s \in \R$ where $\bar{s}_p := (n-1)|1/2 - 1/p|$. Here $L^p_{s}(M)$ denotes the standard Sobolev (or Bessel potential) space on $M$ with Lebesgue exponent $p$ and $s$ derivatives; the relevant definitions are recalled in $\S$3 below. Moreover, provided $t$ avoids a discrete set of times, the estimate \eqref{fixed time estimate} is sharp for all $1 < p < \infty$ in the sense that one cannot replace $\bar{s}_p$ with $\bar{s}_p - \sigma$ for any $\sigma > 0$.

The purpose of this article is to prove sharp \emph{local smoothing} estimates for the solution $u$ for a partial range of $p$, which demonstrate a gain in regularity for space-time estimates over the fixed-time case.

\begin{theorem}\label{local smoothing theorem} If $u$ is the solution to the Cauchy problem \eqref{wave equation} and $\bar{p}_n \leq p < \infty$ where $\bar{p}_n := \tfrac{2(n+1)}{n-1}$, then
\begin{equation}\label{local smoothing estimate}
\Big(\int_1^2\|u(\,\cdot\,,t)\|_{L^p_{s - \bar{s}_p + \sigma}(M)}^p\,\ud t\Big)^{1/p} \lesssim_{M,g} \|f_0\|_{L^p_{s}(M)} + \|f_1\|_{L^p_{s-1}(M)}
\end{equation}
holds for all $s \in \R$ and all $\sigma < 1/p$.
\end{theorem}

For the given range of $p$ this result is sharp up to the endpoint in the sense that the inequality fails if $\sigma > 1/p$.\footnote{Such inequalities are also conjectured to hold at the endpoint (that is, the case $\sigma = 1/p$) and endpoint estimates have been obtained for a further restricted range of $p$ in high-dimensional cases: see \cite{Heo2011} and \cite{Lee2013}.} It is likely, however, that the range of $p$ is not optimal. For instance, Minicozzi and the third author \cite{Minicozzi1997} (see also \cite{Sogge}) found specific manifolds for which \eqref{local smoothing estimate} can hold for all $\sigma < 1/p$ only if $p \geq \frac{2(3n+1)}{3n-3}$ for $n$ odd or $p \geq \frac{2(3n+2)}{3n-2}$ for $n$ even; it is not unreasonable to speculate that these necessary conditions should, for general $M$, be sufficient.\footnote{The examples in \cite{Minicozzi1997} concern certain oscillatory integral operators of Carleson--Sj\"olin type, defined with respect to the geodesic distance on $M$. Their results lead to counterexamples for local smoothing estimates via a variant of the well-known implication ``local smoothing $\Rightarrow$ Bochner--Riesz". Implications of this kind will be discussed in detail in $\S$\ref{section: sharp examples}.} The examples of \cite{Minicozzi1997} rely on Kakeya compression phenomena for families of geodesics; the (euclidean) Kakeya conjecture, if valid, would preclude such behaviour over $\R^n$. Indeed, the \emph{local smoothing conjecture} for the wave equation \cite{Sogge1991} asserts that in the euclidean case the estimate \eqref{local smoothing estimate} should hold for all $\sigma < 1/p$ in the larger range $\tfrac{2n}{n-1} \leq p < \infty$. If true, this would be a remarkable result, not least because the conjecture formally implies many other major open problems in harmonic analysis (including the Bochner--Riesz, Fourier restriction and Kakeya conjectures): see \cite{Tao1999}.

It is well known (see, for instance, \cite[Chapter 5]{Duistermaat2011} or \cite[Chapter 4]{Sogge2017}) that the solution $u$ to the Cauchy problem \eqref{wave equation} is given by
\begin{equation}\label{solution formula}
u(x,t) = \mathcal{F}_{0}f_0(x,t) + \mathcal{F}_{1}f_1(x,t) 
\end{equation}
where, using the language of \cite{Hormander1971} and \cite{Mockenhaupt1993}, each $\mathcal{F}_{j} \in I^{j-1/4}(M \times \R, M; \mathcal{C})$ is a Fourier integral operator (FIO) with canonical relation $\mathcal{C}$ satisfying the cinematic curvature condition (the relevant definitions will be recalled below in $\S$\ref{section: proof of local smoothing}; see also \cite{BHS} for a comprehensive introduction to FIOs in the context of local smoothing.). In local coordinates, such operators $\mathcal{F}_j$ adopt the explicit form \eqref{explicit FIO} below with $\mu=j$. Theorem~\ref{local smoothing theorem} follows from a more general result concerning Fourier integral operators. 

\begin{theorem}\label{FIO local smoothing theorem} Let $n \geq 2$ and let $Y$ and $Z$ be precompact manifolds of dimensions $n$  and $n+1$, respectively. Suppose that $\mathcal{F} \in I^{\mu-1/4}(Z, Y; \mathcal{C})$ where the canonical relation $\mathcal{C}$ satisfies the cinematic curvature condition. If $\bar{p}_n \leq p < \infty$, then 
$$\|\mathcal{F} f\|_{L^p_{\mathrm{loc}}(Z)} \lesssim \| f \|_{L^p_{\mathrm{comp}}(Y)}$$
holds whenever $\mu < -\bar{s}_p + 1/p$. 
\end{theorem}

An interesting feature of Theorem~\ref{FIO local smoothing theorem} is that both the restriction on $\mu$ and the restriction on $p$ is sharp in certain cases. 

\begin{proposition}\label{sharpness proposition} For all odd dimensions $n \geq 3$ there exists some operator $\mathcal{F} \in I^{-(n-1)/2-1/4}(\R^{n+1}, \R^n; \mathcal{C})$ with $\mathcal{C}$ satisfying the cinematic curvature condition such that 
\begin{equation*}
\bigl\| \, (I-\Delta_x)^{\gamma/2} \circ {\mathcal F} f\, \bigr\|_{ L^p(\R^{n+1})} \lesssim \| f \|_{L^p(\R^n)} \quad \textrm{for all $0<\gamma<n/p$}
\end{equation*}
fails for $p < \bar{p}_n$.
\end{proposition}

If $\mathcal{F} \in I^{-(n-1)/2-1/4}(\R^{n+1}, \R^n; \mathcal{C})$, then $(I-\Delta_x)^{\gamma/2} \circ {\mathcal F}  \in I^{\mu-1/4}(\R^{n+1}, \R^n; \mathcal{C})$ for $\mu = -(n-1)/2 + \gamma$ by the composition theorem for Fourier integral operators (see, for instance,  \cite[Theorem 6.2.2]{Sogge2017}). The range $0<\gamma<n/p$ corresponds to $-(n-1)/2 < \mu < -\bar{s}_p + 1/p$ and thus Proposition~\ref{sharpness proposition} demonstrates that Theorem~\ref{FIO local smoothing theorem} is sharp in odd dimensions. 

Proposition~\ref{sharpness proposition} is established by relating local smoothing estimates for Fourier integral operators to $L^p$ estimates for oscillatory integral operators with non-homogeneous phase (sometimes referred to as \emph{H\"ormander-type operators}) and then invoking well-known examples of Bourgain \cite{Bourgain1991, Bourgain1995} for the oscillatory integral problem. The details of the argument are discussed in $\S$\ref{section: sharp examples}.\footnote{It is remarked that the $\mathcal{F}$ constructed to provide sharp examples for Theorem~\ref{FIO local smoothing theorem} do not arise as solutions to wave equations of the kind discussed above. Thus, these examples \emph{do not} show sharpness in Theorem~\ref{local smoothing theorem}. Indeed, it is likely that Theorem~\ref{local smoothing theorem} should hold in the range suggested by the work of Minicozzi and the third author \cite{Minicozzi1997}, as described above (see also the discussion in $\S$\ref{section: sharp examples}).}

At this juncture some historical remarks are in order. Local smoothing estimates for the euclidean wave equation were introduced by the third author in \cite{Sogge1991} and then further investigated in \cite{Mockenhaupt1992}. These early results, however, did not involve a sharp gain in regularity (that is, a sharp range of $\sigma$, at least up to the endpoint); the first sharp local smoothing estimates were established in $\R^2$ in the seminal work of Wolff \cite{Wolff2000}. For this, Wolff \cite{Wolff2000} introduced what have since become known as \emph{decoupling inequalities} for the light cone. The results of \cite{Wolff2000} were improved and extended by a number of authors \cite{Laba2002, Garrigos2009, Garrigos2010, Bourgain2013} before the remarkable breakthrough of Bourgain--Demeter \cite{Bourgain2015} established essentially sharp decoupling estimates in all dimensions (see also \cite{BourgainSF, TV2, Heo2011, Lee2012, Lee} for alternative approaches to the local smoothing problem and \cite{Cladek} for recent work in a related direction). One of the many consequences of the theorem of Bourgain--Demeter \cite{Bourgain2015} is the analogue of Theorem~\ref{local smoothing theorem} for the wave equation in $\R^n$. 

Local smoothing estimates were studied in the broader context of Fourier integral operators in parallel to the developments described above \cite{Mockenhaupt1993, Lee2013} (see also \cite{Sogge2017}). Results in this vein typically follow from variable coefficient extensions of methods used to study wave equations on flat space. Similarly, Theorem~\ref{FIO local smoothing theorem} (and therefore Theorem~\ref{local smoothing theorem}) is a consequence of a natural variable coefficient extension of the decoupling inequality of Bourgain--Demeter \cite{Bourgain2015}. The variable coefficient decoupling theorem is the main result of this paper and concerns certain oscillatory integral operators with homogeneous phase; the setup is described in the following subsection. 

%%%%%%%%%%%%%%%%%%%%%%%%%%%%%%%%%%%%%%%%%%%%%%%%%%%%%%%%%%%%%%%%%%%%%%%%%%%%%%%%%%%%%%%%%%%%%%%%

%                               VARIABLE COEFFICIENT DECOUPLING

%%%%%%%%%%%%%%%%%%%%%%%%%%%%%%%%%%%%%%%%%%%%%%%%%%%%%%%%%%%%%%%%%%%%%%%%%%%%%%%%%%%%%%%%%%%%%%%%

\subsection{Variable coefficient decoupling}\label{subsection: variable coefficient decoupling} Let $a = a_1 \otimes a_2 \in C^{\infty}_c(\R^{n+1} \times \R^n)$ where $a_1 \in C^{\infty}_c(\R^n)$ is supported in $B(0,1)$ and $a_2$ is supported in the domain
\begin{equation*}
\Gamma_1 := \big\{\xi \in \hat{\R}^n : 1/2 \leq \xi_n \leq 2 \textrm{ and } |\xi_j| \leq |\xi_n| \textrm{ for $1 \leq j \leq n-1$}\big\}.
\end{equation*}
Suppose that $\phi \colon \R^n \times \R \times \hat{\R}^n \to \R$ is smooth away from $\R^n \times \R \times \{0\}$ and that for all $(x,t) \in \R^n \times \R$ the function $\xi \mapsto \phi(x,t;\xi)$ is homogeneous of degree 1. Writing $\mathrm{supp}\,a \setminus 0 $ for the set $(\mathrm{supp}\,a) \setminus (\R^n \times \R \times \{0\})$, assume, in addition, that $\phi$ satisfies the following geometric conditions:
\begin{itemize}
\item[H1)]$\mathrm{rank}\, \partial_{ \xi z}^2 \phi(x,t;\xi) = n$ for all $(x,t; \xi) \in \mathrm{supp}\,a \setminus 0$. Here and below $z$ is used to denote vector in $\R^{n+1}$ comprised of the space-time variables $(x,t)$. 
\item[H2)] Defining the generalised Gauss map by $G(z; \xi) := \frac{G_0(z;\xi)}{|G_0(z;\xi)|}$ for all $(z; \xi) \in \mathrm{supp}\,a \setminus 0$ where
\begin{equation*}
G_0(z;\xi) :=  \bigwedge_{j=1}^{n} \partial_{\xi_j} \partial_z\phi(z;\xi),
\end{equation*}
one has
\begin{equation*}
\mathrm{rank}\,\partial^2_{\eta \eta} \langle \partial_z\phi(z;\eta),G(z; \xi)\rangle|_{\eta = \xi} = n-1
\end{equation*}
for all $(z; \xi) \in \mathrm{supp}\,a \setminus 0$.
\end{itemize}

Here the wedge product of $n$ vectors in $\R^{n+1}$ is associated with a vector in $\R^{n+1}$ in the usual manner. It is remarked that H1) and H2) are the natural homogeneous analogues of the Carleson--Sj\"olin \cite{Carleson1972} or H\"ormander \cite{Hormander1973} conditions for non-homogeneous phase functions.

The conditions H1) and H2) naturally arise in the study of Fourier integral operators of the type described in the previous subsection. Indeed, by standard theory (see, for instance, \cite[Proposition 6.1.4]{Sogge2017}), any operator belonging to the class $I^{\mu - 1/4}(Z,Y; \mathcal{C})$ with $\mathcal{C}$ satisfying the cinematic curvature condition can be written in local coordinates as a finite sum of operators of the form
\begin{equation}\label{explicit FIO}
\mathcal{F}f(x,t) := \int_{\hat{\R}^n} e^{i \phi(x,t; \xi)} b(x,t;\xi)(1 + |\xi|^2)^{\mu/2} \hat{f}(\xi)\,\ud \xi
\end{equation}
where  $b$ is a symbol of order 0 (with compact support in the $(x,t)$ variables) and  $\phi$ satisfies the properties H1) and H2) (at least on the support of $b$).

Rather than  directly studying the operators $\mathcal{F}$ as in \eqref{explicit FIO}, a decoupling inequality shall instead be formulated in terms of a certain closely related class of oscillatory integral operators. 

Given $\lambda \geq 1$, define the rescaled phase and amplitude 
\begin{equation*}
\phi^{\lambda}(x,t;\omega) := \lambda\phi(x/\lambda, t/\lambda;\omega) \qquad \textrm{and} \qquad a^{\lambda}(x,t; \xi) := a_1(x/\lambda, t/\lambda) a_2(\xi)
\end{equation*}
and, with this data, let
\begin{equation*}
 T^{\lambda}f(x,t) := \int_{\hat{\R}^{n}} e^{i \phi^{\lambda}(x, t;\xi)} a^{\lambda}(x,t;\xi) f(\xi) \,\ud \xi.
\end{equation*}
The aforementioned variable coefficient decoupling inequality compares the $L^p$-norm of $T^{\lambda}f$ with the $L^p$-norms of localised pieces $T^{\lambda}f_{\theta}$ which form a decomposition of the original operator. To describe this decomposition fix a second spatial parameter $1 \leq R \leq \lambda$ and note that the support of $a_2$ intersects the affine hyperplane $\xi_n = 1$ on the disc $B^{n-1}(0,1) \times \{1\}$. Fix a maximally $R^{-1/2}$-separated subset of $B^{n-1}(0,1) \times \{ 1\}$ and for each $\omega$ belonging to this subset define the $R^{-1/2}$-\emph{plate}
\begin{equation*}
\theta := \big\{(\xi',\xi_n) \in \hat{\R}^n : 1/2 \leq \xi_n \leq 2 \textrm{ and } |\xi'/\xi_n - \omega| \leq R^{-1/2} \big\}. 
\end{equation*}
In this case $(\omega,1) \in B^{n-1}(0,1) \times \{1\}$ is referred to as the \emph{centre} of the $R^{-1/2}$-plate $\theta$. Thus, the collection of all $R^{-1/2}$-plates forms a partition of the support of $a_2$ into finitely-overlapping subsets (see Figure~\ref{plate diagram}). For each $\theta$, let $\widetilde{\theta}$ be a subset of $\theta$ such that the family of all $\widetilde{\theta}$ forms a partition of the support of $a_2$. Given any function $f \in L^1_{\mathrm{loc}}(\hat{\R}^n)$ and an $R^{-1/2}$-plate $\theta$, define $f_{\theta} := \chi_{\widetilde{\theta}}f$, and for $1 \leq p < \infty$ and any measurable set $E \subseteq \R^{n+1}$ introduce the \emph{decoupled norm}
\begin{equation*}
\|T^{\lambda}f\|_{L^{p,R}_{\mathrm{dec}}(E)} := \big(\sum_{\theta : R^{-1/2}-\mathrm{plate}} \|T^{\lambda}f_{\theta}\|_{L^{p}(E)}^p \Big)^{1/p}.
\end{equation*}
This definition is extended to the case $p = \infty$ and to weighted norms  $\|T^{\lambda}f\|_{L^{p,R}_{\mathrm{dec}}(w)}$ in the obvious manner.
\begin{figure}
\begin{center}
\resizebox {0.95\textwidth} {!} {\begin{tikzpicture}

 \fill[fill=blue!10] (2.5,2.5) -- (10,10) -- (-10, 10) -- (-2.5, 2.5) -- cycle;			
		
 \fill[fill=yellow!30] (0.5,2.5) -- (2,10) -- (4, 10) -- (1, 2.5) -- cycle;		   
    
		\draw[black, very thick] 
		(-12.5,0) -- (12.5,0);
		\draw[black, very thick] 
		(0,-2) -- (0,13.2);
		
		\draw[black, very thick] 
    (2.5,2.5) -- (10,10);
		\draw[black, very thick] 
		(-2.5,2.5) -- (-10,10);

		\draw[black, very thick] 
		(-10,10) -- (10,10);
		\draw[black, thick, dashed] 
		(-12.5,10) -- (-10,10);
		\draw[black, thick, dashed] 
		(10,10) -- (12.5,10);
		
		\draw[black, very thick] 
		(-2.5,2.5) -- (2.5,2.5);
    \draw[black, thick, dashed] 
		(-12.5,2.5) -- (-2.5,2.5);
		\draw[black, thick, dashed] 
		(2.5,2.5) -- (12.5,2.5);

\foreach \a in {-10,-8,...,10}
    {
		\draw[black, thick, dashed] 
		({\a * 0.25},2.5) -- (0,0);
		\draw[black, thick, dashed] 
		({\a},{10}) -- ({\a * 1.25},{12.5});
		\draw[black, thick] 
		({\a},{10}) -- ({\a * 0.25},2.5);
		
		}
		
\node[left, scale=2] at  (-8, 7) {\Huge$\Gamma_1$};
\node[left, scale=1.5] at  (2.6, 7) {\Huge$\theta$};
\node[below, scale=1.5] at  (11.8, 0) {\Huge$\xi'$};
\node[left, scale=1.5] at  (0, 12.5) {\Huge$\xi_n$};
\node[below, scale=1.5] at  (-11, 2.5) {\Huge$\xi_n = 1/2$};
\node[below, scale=1.5] at  (-11.8, 10) {\Huge$\xi_n=2$};

\filldraw[black](1.5,5) circle (3pt);
\draw[black, thick, dashed] 
		(-5,5) -- (5,5);
\node[right, scale=1.5] at  (8, 4.7) {\Huge$(\omega,1)$};		
\path[red,line width=1.3pt,-{>[scale=1.5, length=10, width=10]}] (8,4.7) edge [bend left] (1.6,4.8);		
		\end{tikzpicture}}
\caption{The decomposition of the domain $\Gamma_1$ into $R^{-1/2}$-plates. The centre $(\omega,1)$ of one such plate $\theta$ is indicated.}
\label{plate diagram}
\end{center}
\end{figure}
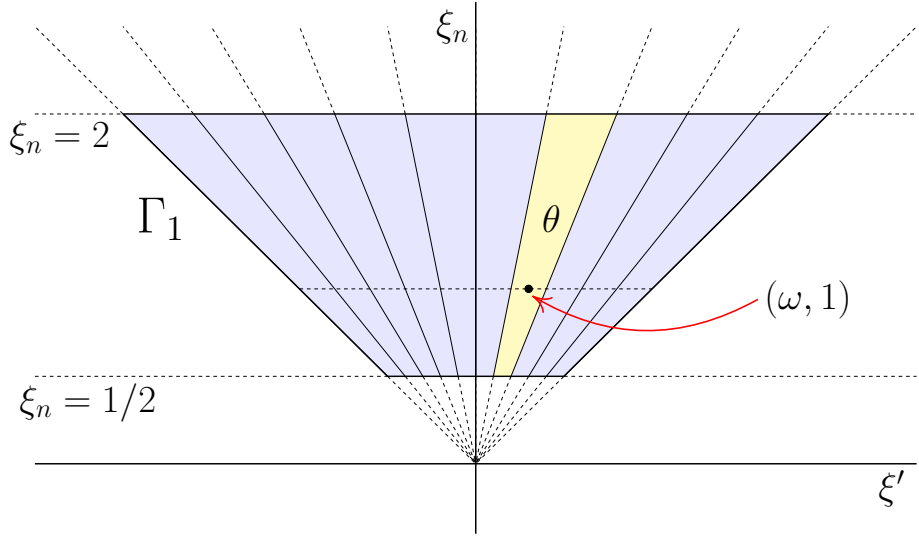

Finally, let $\bar{p}_n$ and $\bar{s}_p$ be as in the statement of Theorem~\ref{local smoothing theorem} and given $2 \leq p \leq \infty$ define the exponent 
\begin{equation}\label{alpha exponent}
\alpha(p) := \left\{ \begin{array}{ll}
\bar{s}_p/2 & \textrm{if $2 \leq p \leq \bar{p}_n$}, \\
\bar{s}_p - 1/p & \textrm{if $\bar{p}_n \leq p \leq \infty$.}
\end{array}\right. 
\end{equation} 
With these definitions, the decoupling theorem reads as follows.

\begin{theorem}\label{decoupling theorem} Let $T^{\lambda}$ be an operator of the form described above and $2 \leq p \leq \infty$. For all $\varepsilon > 0$ and $M \in \N$ one has\footnote{Strictly speaking, the proof will establish this inequality with the operator appearing on the right-hand side of \eqref{decoupling theorem equation} defined with respect to an amplitude with slightly larger spatial support than that appearing in the operator on the left (but both operators are defined with respect to the same phase function). This has no bearing on the applications and such slight discrepancies will be suppressed in the notation.}
\begin{equation}\label{decoupling theorem equation}
\|T^{\lambda}f\|_{L^p(\R^{n+1})} \lesssim_{\varepsilon, M, \phi, a} \lambda^{\alpha(p) +\varepsilon} \|T^{\lambda}f\|_{L^{p,\lambda}_{\mathrm{dec}}(\R^{n+1})}  + \lambda^{-M}\|f\|_{L^2(\hat{\R}^n)}.
\end{equation}
\end{theorem}

Theorem~\ref{decoupling theorem} is a natural variable coefficient extension of (the $\ell^p$ variant of) Theorem 1.2 in \cite{Bourgain2015}, which treats the prototypical case $\phi(x,t;\xi) = \langle x, \xi \rangle + t |\xi|$. More generally, the \emph{translation-invariant} case, where $\phi$ is linear in the variables $x,t$, can be deduced from the results of \cite{Bourgain2015, Bourgain2017a} via an argument originating in \cite{Pramanik2007, Garrigos2010}. Interestingly, it transpires that the result for general operators $T^{\lambda}$ follows itself from the translation-invariant case. This stands in contrast with the $L^p$-theory of such operators (see, for instance, \cite{Bourgain2011, Guth}). 

Finally, it is remarked that the argument used to prove Theorem~\ref{decoupling theorem} is flexible in nature, and could equally be applied to prove natural variable coefficient extensions of other known decoupling results, such as the $\ell^2$ decoupling theorem for the paraboloid \cite{Bourgain2015} or the decoupling theorem of Bourgain--Demeter--Guth \cite{Bourgain2016} for the moment curve (in the latter case the relevant variable coefficient operators are those appearing in \cite{Bak2004, Bak2009}).

\begin{acknowledgement} The first and second authors are indebted to Andreas Seeger for helpful conversations and his hospitality during their visit to UW-Madison in October, 2017. They would also like to thank Jon Bennett, Stefan Buschenhenke and Ciprian Demeter for useful and encouraging discussions on decoupling inequalities. 
\end{acknowledgement}

%%%%%%%%%%%%%%%%%%%%%%%%%%%%%%%%%%%%%%%%%%%%%%%%%%%%%%%%%%%%%%%%%%%%%%%%%%%%%%%%%%%%%%%%%%%%%%%%

%                                A  PROOF OF THE  VARIABLE COEFFICIENT DECOUPLING INEQUALITY

%%%%%%%%%%%%%%%%%%%%%%%%%%%%%%%%%%%%%%%%%%%%%%%%%%%%%%%%%%%%%%%%%%%%%%%%%%%%%%%%%%%%%%%%%%%%%%%%

\section{A proof of the variable coefficient decoupling inequality}

\subsection{An overview of the proof} As indicated in the introduction, Theorem~\ref{decoupling theorem} will be derived as a consequence of the (known) translation-invariant case; the latter result is recalled presently. Let $a_2$ be as in the introduction and suppose $h \colon \hat{\R}^n \to \R$ is smooth away from 0, homogeneous of degree 1 and satisfies $\mathrm{rank}\,\partial_{\xi\xi}^2h(\xi) = n-1$ for all $\xi \in \mathrm{supp}\,a_2\setminus\{0\}$. With this data, define the extension operator
\begin{equation*}
Ef(x,t) := \int_{\hat{\R}^n} e^{i(\langle x, \xi \rangle + t h(\xi))}a_2(\xi)f(\xi)\,\ud \xi. 
\end{equation*} 
For the exponent $\alpha$ defined in \eqref{alpha exponent}, the translation-invariant case of the theorem (due to Bourgain--Demeter \cite{Bourgain2015, Bourgain2017a}) reads thus.

\begin{theorem}[Bourgain--Demeter \cite{Bourgain2015, Bourgain2017a}]\label{translation-invariant decoupling theorem} For all $2 \leq p \leq \infty$ and all $\varepsilon > 0$ the estimate
\begin{equation}\label{translation-invariant decoupling inequality}
\|Ef\|_{L^p(w_{B_{\lambda}})} \lesssim_{\varepsilon, N, h, a} \lambda^{\alpha(p) +\varepsilon} \|Ef\|_{L^{p,\lambda}_{\mathrm{dec}}(w_{B_{\lambda}})}
\end{equation}
holds for $\lambda \geq 1$. 
\end{theorem}

Here $B_{R}$ denotes a ball of radius $R$ for any $R > 0$ and $w_{B_{R}}$ is a rapidly decaying weight function, concentrated on $B_{R}$. In particular, if $(\bar{x},\bar{t}) \in \R^n \times \R$ denotes the centre of $B_{R}$, then
\begin{equation}\label{weight function}
w_{B_{R}}(x,t) := \big(1 + R^{-1}|x - \bar{x}| + R^{-1} |t - \bar{t}|\big)^{-N}, 
\end{equation}
where $N$ can be taken to be any sufficiently large integer (depending on $n$, $h$ and $p$). It is remarked that the dependence on $h$ of the implicit constant in the inequality \eqref{translation-invariant decoupling inequality} involves only the size of the absolute values of the non-zero eigenvalues of $\partial_{\xi \xi}^2 h$ and their reciprocals, as well as upper bounds for a finite number of higher order derivatives $\partial_{\xi}^{\beta}h$, $|\beta| \geq 3$. 

As mentioned in the introduction, Theorem~\ref{translation-invariant decoupling theorem} does not appear in \cite{Bourgain2015, Bourgain2017a} in the stated generality, but this result may be readily deduced from the prototypical cases considered in \cite{Bourgain2015, Bourgain2017a} via the arguments of \cite{Pramanik2007, Garrigos2010} (see also \cite[$\S\S$7-8]{Bourgain2015} and \cite{Guo}), or by using a variant of the approach developed in the present article.

The passage from Theorem~\ref{translation-invariant decoupling theorem} to Theorem~\ref{decoupling theorem} is, in essence, realised in the following manner. The desired decoupling inequalities have a `self-similar' structure, which is manifested in their almost-invariance under certain parabolic rescaling (see Lemma~\ref{parabolic rescaling lemma}). An implication of this self-similarity is that in order to prove the decoupling estimate, it suffices to obtain some non-trivial, but possibly very small, gain at a single spatial scale; this gain can then be propagated through all the scales via parabolic rescaling.\footnote{Further details and discussion of this perspective on decoupling theory can be found in the recorded lecture series given by Guth as part of the MSRI harmonic analysis programme during January 2017 \cite{Guth_MSRI_1, Guth_MSRI_2, Guth_MSRI_3}.} At spatial scales $K$ below the critical value $\lambda^{1/2}$ one can effectively approximate $T^{\lambda}$ by an extension operator $E$ of the form described above: this is the content of Lemma~\ref{approximation lemma} below. Combining this approximation with Theorem~\ref{translation-invariant decoupling theorem} provides some gain at such scales $K$, and combining these observations concludes the argument.

%%%%%%%%%%%%%%%%%%%%%%%%%%%%%%%%%%%%%%%%%%%%%%%%%%%%%%%%%%%%%%%%%%%%%%%%%%%%%%%%%%%%%%%%%%%%%%%%

%                                        PROPERTIES OF THE PHASE

%%%%%%%%%%%%%%%%%%%%%%%%%%%%%%%%%%%%%%%%%%%%%%%%%%%%%%%%%%%%%%%%%%%%%%%%%%%%%%%%%%%%%%%%%%%%%%%%

\subsection{Basic properties of the phase}\label{subsection: properties of the phase} Before carrying out the programme described above it is useful to note some basic properties of homogeneous phases $\phi$ satisfying the conditions H1) and H2) and to make some simple reductions. 

After a localisation and a translation argument, one may assume that $a$ is supported inside $Z \times \Xi$ where $Z := X \times T$ for $X \subseteq B(0,1) \subseteq \R^{n}$ and $T \subseteq (-1,1) \subseteq \R$  small open neighbourhoods of the origin and $\Xi \subseteq \Gamma_1$ is a small open sector around $e_n := (0, \dots, 0,1) \in \hat{\R}^n$. By choosing the size of the neighbourhoods appropriately, one may assume the phase satisfies a number of useful additional properties, described presently. 

By localising, one may ensure that strengthened versions of the conditions H1) and H2) hold. In particular, without loss of generality one may work with phases satisfying:
\begin{itemize}
\item[H1$^\prime$)] $\det \partial_{ \xi x}^2 \phi(z;\xi) \neq 0$ for all $(z; \xi) \in Z \times \Xi$;
\item[H2$^\prime$)] $\det \partial_{\xi' \xi'} \partial_t \phi (z;\xi) \neq 0$ for all $(z; \xi) \in Z \times \Xi$.
\end{itemize}

Indeed, by precomposing the phase with a rotation in the $z = (x,t)$ variables, one may assume that $G(0;e_n)=e_{n+1}$ and therefore $\partial_\xi \partial_t \phi (0; e_n)=0$. Hence, by H1), it follows that $\det \partial_{ \xi x}^2 \phi(0;e_n) \neq 0$. On the other hand, by the homogeneity of $\phi$, every $(n-1)\times(n-1)$ minor of the matrix featured in the H2) condition is a multiple of $\det \partial_{\eta'\eta'} \langle \partial_z \phi(z;\eta), G(z; \xi)\rangle|_{\eta = \xi}$. Thus, in order for the rank condition H2) to hold, this determinant must be non-zero. In particular, as $G(0;e_n)=e_{n+1}$, it follows that $\det \partial_{\xi' \xi'} \partial_t \phi (0;e_n) \neq 0$. Choosing the neighbourhoods $Z$ and $\Xi$ sufficiently small now ensures both H1$^\prime$) and H2$^\prime$) hold.

By Euler's homogeneity relations,
\begin{equation*}
\partial_x \phi (x,t;\xi) = \sum_{j=1}^n \xi_j \cdot \partial_{\xi_j}\partial_x \phi (x,t;\xi).
\end{equation*}
It follows that for each $t \in (-1,1)$ and $\xi \in \hat{\R}^n$ the Jacobian determinant of the map $x \mapsto ((\partial_{\xi'}\phi)(x,t;\xi), \phi(x,t;\xi))$ is given by $\xi_n \cdot \det \partial_{\xi x}^2\phi(x,t;\xi)$, which is non-zero by H1$^\prime$). Thus, there exists a smooth local inverse mapping $\Upsilon(\,\cdot\,,t; \xi)$ which satisfies
\begin{equation}\label{Upsilon}
(\partial_{\xi'}\phi)(\Upsilon(y,t; \xi),t;\xi) = y' \quad \textrm{and} \quad \phi(\Upsilon(y,t; \xi),t;\xi) = y_n. 
\end{equation}
Similarly, there exists a smooth mapping  $\Psi(x,t; \,\cdot\,)$ such that
\begin{equation}\label{Psi}
(\partial_{x}\phi)(x,t;\Psi(x,t; \eta)) = \eta.
\end{equation}
Given $\lambda \geq 1$, let $\Upsilon^{\lambda}$ and $\Psi^{\lambda}$ denote the natural rescaled versions of these maps, so that $\Upsilon^{\lambda}(z;\xi) = \lambda \Upsilon(y/\lambda;\xi)$ and $\Psi^{\lambda}(z; \eta) := \Psi(z/\lambda; \eta)$. One may assume that $Z$ and $\Xi$ are such that the above mappings are everywhere defined.

%%%%%%%%%%%%%%%%%%%%%%%%%%%%%%%%%%%%%%%%%%%%%%%%%%%%%%%%%%%%%%%%%%%%%%%%%%%%%%%%%%%%%%%%%%%%%%%%

%                                        QUANTITATIVE CONDITIONS

%%%%%%%%%%%%%%%%%%%%%%%%%%%%%%%%%%%%%%%%%%%%%%%%%%%%%%%%%%%%%%%%%%%%%%%%%%%%%%%%%%%%%%%%%%%%%%%%

\subsection{Quantitative conditions} Fix $\varepsilon > 0$, $M \in \N$ and $2 \leq p < \infty$ (the $p=\infty$ case of Theorem~\ref{decoupling theorem} is trivial but nevertheless must be treated separately: see \eqref{trivial decoupling}). To facilitate certain induction arguments, it is useful to work with quantitative versions of the conditions H1$^\prime$) and H2$^\prime$) on the phase function. In particular, let $c_\mathrm{par}$ be a small fixed constant and assume that for some $0 \leq \sigma_+ \leq n-1$ and $\mathbf{A}=(A_1,A_2,A_3) \in [1,\infty)^3$ the phase satisfies, in addition to H1$^\prime$) and H2$^\prime$), the following:
\begin{itemize}
\item[H1$_{\mathbf{A}}$)] $| \partial_{ \xi x}^2\phi(z;\xi) - I_n| \leq c_{\mathrm{par}}A_1$ for all $(z;\xi) \in Z \times \Xi$. 
\item[H2$_{\mathbf{A}}$)] $|\partial^2_{\xi' \xi'}\partial_t \phi(z;\xi) - \frac{1}{\xi_n} I_{n-1, \sigma_+}| \leq c_{\mathrm{par}} A_2$ for all $(z;\xi) \in Z \times \Xi$, where 
\begin{equation*}
I_{n-1,\sigma_+}:=\mathrm{diag}(\underbrace{1,\dots, 1}_{\sigma_+},\underbrace{ -1, \dots, -1}_{n-1-\sigma_+})
\end{equation*}
 is an $(n-1) \times (n-1)$  diagonal matrix.
\end{itemize}
Some additional control on the size of various derivatives, which is of a rather technical nature, is assumed:
\begin{itemize}
\item[D1$_{\mathbf{A}})$] $\|\partial_{\xi}^{\beta} \partial_{x_k} \phi\|_{L^{\infty}(Z \times \Xi)} \leq c_{\mathrm{par}} A_1$ for all $1 \leq k \leq n$ and $\beta \in \N_0^n$ with $2 \leq |\beta| \leq 3$ satisfying $|\beta'| \geq 2$;
\newline
$\|\partial_{\xi'}^{\beta'} \partial_{t} \phi\|_{L^{\infty}(Z \times \Xi)} \leq \frac{c_{\mathrm{par}}}{2n} A_1$ for all $\beta' \in \N_0^{n-1}$ with $ |\beta'| = 3$.
\item[D2$_{\mathbf{A}})$] For some large integer $N = N_{\varepsilon, M, p} \in \N$, depending only on the dimension $n$ and the fixed choice of $\varepsilon$, $M$ and $p$, one has
\begin{equation*}
\|\partial_{\xi}^{\beta} \partial_{z}^{\alpha} \phi\|_{L^{\infty}(Z \times \Xi)} \leq \frac{c_{\mathrm{par}}}{2n} A_3
\end{equation*}
for all $(\alpha, \beta) \in \N_0^{n+1} \times \N_0^n$ with $2 \leq |\alpha| \leq 4N$ and $1 \leq |\beta| \leq 4N+2$ satisfying $1 \leq |\beta| \leq 4N$ or $|\beta'| \geq 2$. 
\end{itemize}
Finally, it is useful to assume a \emph{margin} condition on the spatial support of the amplitude $a$:
\begin{itemize}
\item[M$_{\mathbf{A}}$)]  $\mathrm{dist}(\mathrm{supp}\,a_1, \R^{n+1}\setminus Z) \geq 1/4A_3$. 
\end{itemize}

Datum $(\phi, a)$ satisfying H1$_{\mathbf{A}}$), H2$_{\mathbf{A}}$), D1$_{\mathbf{A}}$), D2$_{\mathbf{A}}$) and M$_{\mathbf{A}}$) (in addition to H1$^\prime$) and H2$^\prime$)) is said to be of \emph{type $\mathbf{A}$}. One may easily verify that any phase function satisfying H1$^\prime$) and H2$^\prime$) is of type $\mathbf{A}$ for some $\mathbf{A}=(A_1, A_2, A_3) \in [1, \infty)^3$. The conditions H1$_{\mathbf{A}}$) and H2$_{\mathbf{A}}$) are quantitative substitutes for H1$^\prime$) and H2$^\prime$) if, say, $A_1, A_2 \leq 1$; for $A_1$ and $A_2$ large, however, the conditions H1$_\mathbf{A})$ and H2$_\mathbf{A})$ are vacuous and do not imply H1$^\prime$) or H2$^\prime$). By various rescaling arguments, it is possible to reduce to the case where $\mathbf{A} = \mathbf{1} := (1,1,1)$, as shown in \S\ref{subsection: parabolic rescaling}.

%%%%%%%%%%%%%%%%%%%%%%%%%%%%%%%%%%%%%%%%%%%%%%%%%%%%%%%%%%%%%%%%%%%%%%%%%%%%%%%%%%%%%%%%%%%%%%%%

%                                        SETTING UP THE INDUCTION

%%%%%%%%%%%%%%%%%%%%%%%%%%%%%%%%%%%%%%%%%%%%%%%%%%%%%%%%%%%%%%%%%%%%%%%%%%%%%%%%%%%%%%%%%%%%%%%%

\subsection{Setting up the induction for \eqref{decoupling theorem equation} and reduction to $\lambda^{1-\varepsilon/n}$-balls} Continuing with the fixed $\varepsilon$, $M$ and $p$ from the previous subsection, let $\mathbf{A}=(A_1,A_2,A_3) \in [1,\infty)^3$ and $N \in \N$ be as in the definition of the condition D2$_{\mathbf{A}}$). For $1 \leq R \leq \lambda$ let $\mathfrak{D}_{\mathbf{A}}^{\varepsilon}(\lambda; R)$ denote the infimum over all $C \geq 0$ for which the inequality
\begin{equation}\label{decoupling constant definition}
\|T^{\lambda}f\|_{L^p(B_R)} \leq CR^{\alpha(p) + \varepsilon}  \|T^{\lambda}f\|_{L^{p,R}_{\mathrm{dec}}(w_{B_R})} + R^{2n} (\lambda/R)^{-\varepsilon N/8n} \|f\|_{L^2(\hat{\R}^n)}
\end{equation}
holds for all type $\mathbf{A}$ data $(\phi, a)$\footnote{As in the statement of Theorem~\ref{decoupling theorem}, a discrepancy between the amplitude functions is allowed here: the right-hand operator is understood to be defined with respect to \emph{some} amplitude with possibly slightly larger spatial support than the original amplitude $a$.} and balls $B_R$ of radius $R$ contained in $B(0,\lambda)$. Here the weight function is understood to be defined with respect to the fixed choice of $N$ above, as in \eqref{weight function}. It is remarked that the quantity $\mathfrak{D}_{\mathbf{A}}^{\varepsilon}(\lambda; R)$ is always finite. To see this, note that for any $1 \leq \rho \leq R$ and $\rho^{-1/2}$-plate $\theta$ one may write 
\begin{equation*}
T^\lambda f_\theta = \sum_{\substack{\sigma \cap \widetilde{\theta} \neq \emptyset \\ \sigma: R^{-1/2}-\mathrm{plate}}} T^\lambda f_\sigma;
\end{equation*}
recall that $\widetilde{\theta}$ is the subset of $\theta$ upon which $f_\theta$ is supported. By the triangle and H\"older's inequalities, for any weight $w$ one has 
\begin{equation}\label{trivial decoupling 0}
 \|T^{\lambda}f\|_{L^{p, \rho}_{\mathrm{dec}}(w)} \leq (R/\rho)^{(n-1)/2p'}\|T^{\lambda}f\|_{L^{p, R}_{\mathrm{dec}}(w)}.
\end{equation}
Taking $\rho$ = 1, one thereby deduces the trivial bound 
\begin{equation}\label{trivial decoupling}
\mathfrak{D}_{\mathbf{A}}^{\varepsilon}(\lambda; R) \lesssim R^{(n-1)/2p' - \alpha(p)},
\end{equation} 
which, in particular, shows that $\mathfrak{D}_{\mathbf{A}}^{\varepsilon}(\lambda; R)$ is finite. This trivial observation also proves Theorem~\ref{decoupling theorem} in the $p = \infty$ case. 

To prove Theorem~\ref{decoupling theorem} for the fixed parameters $2\varepsilon$, $M$ and $2 \leq p < \infty$ it is claimed that it suffices to show that
\begin{equation}\label{away from lambda}
\mathfrak{D}_{\mathbf{A}}^{\varepsilon}(\lambda; \lambda^{1-\varepsilon/n}) \lesssim_{\mathbf{A},\varepsilon} 1.
\end{equation}
The `$\varepsilon/n$-gain' realised by this reduction will be useful for various technical reasons. To see the above claim, observe that the support conditions on the amplitude $a$ imply that the support of $T^\lambda f$ is always contained in $B(0,\lambda)$. Take a cover of $B(0,\lambda)$ by finitely-overlapping $\lambda^{1-\varepsilon/n}$-balls and apply \eqref{away from lambda} to the relevant $L^p$-norm defined over each of these balls. Summing over all the contributions from the collection via Minkowski's inequality, one deduces that 
\begin{equation*}
\|T^{\lambda}f\|_{L^p(B(0,\lambda))} \lesssim_{\mathbf{A}, \varepsilon} \lambda^{\alpha(p) + \varepsilon} \|T^{\lambda} f\|_{L^{p,\lambda^{1-\varepsilon/n}}_{\mathrm{dec}}(w_{B(0,\lambda)})} +\lambda^{2n-\varepsilon N/8n} \|f\|_{L^2(\hat{\R}^n)}.
\end{equation*}
Here the weight $w_{B(0,\lambda)}$ is as defined in \eqref{weight function} (with $R = \lambda$ and $\bar{x} = 0$, $\bar{t} = 0$). Provided $N$ is sufficiently large, the desired estimate \eqref{decoupling theorem equation} now follows from \eqref{trivial decoupling 0}. 

After reducing to the case $\mathbf{A} = \mathbf{1}$, it will be shown in $\S$\ref{subsec:proof}, using induction on $R$, that $\mathfrak{D}_{\mathbf{1}}^{\varepsilon}(\lambda; R) \lesssim_{\varepsilon} 1$ for all $1 \leq R \leq \lambda^{1-\varepsilon/n}$, thus establishing \eqref{away from lambda}. The trivial inequality \eqref{trivial decoupling} will serve as the base case for this induction.

%%%%%%%%%%%%%%%%%%%%%%%%%%%%%%%%%%%%%%%%%%%%%%%%%%%%%%%%%%%%%%%%%%%%%%%%%%%%%%%%%%%%%%%%%%%%%%%%

%                                          PARABOLIC RESCALING

%%%%%%%%%%%%%%%%%%%%%%%%%%%%%%%%%%%%%%%%%%%%%%%%%%%%%%%%%%%%%%%%%%%%%%%%%%%%%%%%%%%%%%%%%%%%%%%%

\subsection{Parabolic rescaling} \label{subsection: parabolic rescaling}  The first ingredient required in the proof of Theorem~\ref{decoupling theorem} is a standard parabolic rescaling lemma. Before stating this result, it is useful to observe the following trivial consequence of rescaling.

\begin{lemma}\label{basic rescaling lemma}
Let $\mathbf{A}=(A_1,A_2,A_3)$ and $\tilde{\mathbf{A}}=(A_1,A_2,1)$. Then
\begin{equation*}
\mathfrak{D}_{\mathbf{A}}^{\varepsilon}(\lambda; R) \lesssim_{A_3} \mathfrak{D}_{\tilde{\mathbf{A}}}^{\varepsilon}(\lambda/A_3; R/A_3).
\end{equation*}
\end{lemma}
\begin{proof} Let $(\phi, a)$ be a type $\mathbf{A}$ datum. Observe that $T^{\lambda}f = \tilde{T}^{\lambda/A_3}f$ where $\tilde{T}$ is defined with respect to the phase $\tilde{\phi}(z; \xi) := A_3\phi(z/A_3; \xi)$ and amplitude $\tilde{a}(z;\xi) := a(z/A_3; \xi)$. Clearly the datum $(\tilde{\phi}, \tilde{a})$ satisfies H1$_{\tilde{\mathbf{A}}}$), H2$_{\tilde{\mathbf{A}}}$), D1$_{\tilde{\mathbf{A}}})$ and D2$_{\tilde{\mathbf{A}}})$. The margin of the new amplitude $\tilde{a}$ (with respect to the rescaled open set $A_3Z$) has been increased to size 1/4 and so M$_{\tilde{\mathbf{A}}}$) holds. There is a slight issue here in that the support of the rescaled amplitude may now lie outside the unit ball, but one may decompose the amplitude via a partition of unity and translate each piece to write the operator as a sum of $O(A_3^{n+1})$ operators each associated to type $\tilde{\mathbf{A}}$ data. Finally, covering $B(0,R)$ with a union of $R/A_3$-balls and applying the definition of $\mathfrak{D}_{\tilde{\mathbf{A}}}^{\varepsilon}(\lambda/A_3; R/A_3)$ to each of the contributions arising from these balls, the result then follows from the trivial decoupling inequality \eqref{trivial decoupling 0}.
\end{proof}

\begin{lemma}[Parabolic rescaling]\label{parabolic rescaling lemma} Let $1 \leq \rho \leq R \leq \lambda$ and suppose that $T^{\lambda}$ is defined with respect to a type $\mathbf{A}= (A_1,A_2,A_3)$ datum. If $g$ is supported on a $\rho^{-1}$-plate and $\rho$ is sufficiently large depending on $\phi$, then there exists a constant $\bar{C} = \bar{C}_{\phi} \geq 1$ such that
\begin{align}
\label{parabolic rescaling inequality}
\|T^{\lambda}g\|_{L^p(w_{B_{R}})} \lesssim_{\varepsilon, \phi, N} & \;\mathfrak{D}_{\mathbf{1}}^{\varepsilon}(\lambda/\bar{C}\rho^2; R/\bar{C}\rho^2)  (R/\rho^2)^{\alpha(p) + \varepsilon}\|T^{\lambda}g\|_{L^{p,R}_{\mathrm{dec}}(w_{B_R})} \\
\nonumber
& \qquad + R^{2n}(\lambda/R)^{-\varepsilon N/8n}\|g\|_{L^2(\hat{\R}^n)}.
\end{align}
\end{lemma}

\begin{remark}\label{parabolic rescaling remark} The proof of the lemma will show, more precisely, that the lower bound for $\rho$ and the implicit constant in \eqref{parabolic rescaling inequality} may be chosen so as to depend only on $\varepsilon$, $\mathbf{A}$ and the following quantities:
\begin{itemize}
\item $\displaystyle \inf_{(x, t; \xi) \in \mathrm{supp}\,a} |\det \partial_{x \xi}^2 \phi (x,t; \xi)|$. 
\item The infimum and supremum of the magnitudes of the eigenvalues of
\begin{equation}\label{matrix eigenvalues}
\partial_{\xi'\xi'}^2 \partial_t \phi(x,t; \xi)
\end{equation}
over all $(x, t; \xi) \in \mathrm{supp}\,a$. 
\end{itemize}
It is remarked that the quantities appearing in the above bullet points are non-zero by the conditions H1$^\prime$) and H2$^\prime$).
\end{remark}

Lemma~\ref{parabolic rescaling lemma} will be applied in two different ways:
\begin{enumerate}[i)]
 \item An initial application of the lemma reduces the proof of Theorem~\ref{decoupling theorem} to operators defined with respect to type $\mathbf{1}$ data. This is achieved by introducing a partition of unity of the frequency domain $\Gamma_1$ into $\rho^{-1}$-plates for some sufficiently large $\rho$, depending on $\phi$. Each of these frequency-localised pieces can be rescaled via Lemma~\ref{parabolic rescaling lemma} and then summed together to yield the desired reduction. Observe that, by the preceding remark, Lemma~\ref{parabolic rescaling lemma} is uniform for type $\mathbf{1}$ data.
 \item The second application of Lemma~\ref{parabolic rescaling lemma} will be to facilitate an induction argument which constitutes the proof of Theorem~\ref{decoupling theorem} proper. The uniformity afforded by the reduction to type $\mathbf{1}$ phases is useful in order to ensure that this induction closes.
\end{enumerate}

\begin{proof}[Proof (of Lemma~\ref{parabolic rescaling lemma})] Let $\omega \in B^{n-1}(0,1)$ be such that $(\omega,1)$ is the centre of the $\rho^{-1}$-plate upon which $g$ is supported, so that 
\begin{equation*}
\mathrm{supp}\, g \subseteq \big\{ (\xi',\xi_n) \in \hat{\R}^n : 1/2 \leq \xi_n \leq 2 \textrm{ and } |\xi'/\xi_n - \omega| \leq \rho^{-1}  \big\}.
\end{equation*}
Performing the change of variables $ (\xi', \xi_n) = (\eta_n\omega + \rho^{-1}\eta', \eta_n)$, it follows that 
\begin{equation*}
T^{\lambda}g(z) = \int_{\hat{\R}^n} e^{i \phi^{\lambda}(z;\eta_n \omega +  \rho^{-1}\eta', \eta_n)} a^\lambda(z;\eta_n \omega + \rho^{-1} \eta', \eta_n) \tilde{g}(\eta)\,\ud \eta
\end{equation*}
where $\tilde{g}(\eta) := \rho^{-(n-1)}g(\eta_n \omega + \rho^{-1} \eta', \eta_n)$ and $\mathrm{supp}\,\tilde{g} \subseteq \Xi$. 

By applying a Taylor series expansion and using the homogeneity, the phase function in the above oscillatory integral may be expressed as
\begin{equation*}
\phi(z;\omega, 1)\eta_n + \rho^{-1}\langle \partial_{\xi'} \phi(z;\omega, 1), \eta' \rangle + \rho^{-2} \int_0^1 (1-r)\langle \partial_{\xi'\xi'}^2 \phi(z;\eta_n\omega + r\rho^{-1}\eta', \eta_n)\eta', \eta'  \rangle\,\ud r.
\end{equation*}
Let $\Upsilon_{\omega}(y,t) := (\Upsilon(y, t;\omega,1), t)$ and $\Upsilon_{\omega}^{\lambda}(y,t) := \lambda\Upsilon_{\omega}(y/\lambda, t/\lambda)$ and introduce the anisotropic dilations $D_{\rho}(y', y_n, t) := (\rho y', y_n, \rho^2 t)$ and $D_{\rho^{-1}}'(y', y_n) := (\rho^{-1} y', \rho^{-2}y_n)$ on $\R^{n+1}$ and $\R^n$, respectively. Recalling \eqref{Upsilon}, it follows that 
\begin{equation*}
T^{\lambda}g \circ \Upsilon_{\omega}^{\lambda}\circ D_{\rho} = \tilde{T}^{\lambda/\rho^2}\tilde{g}
\end{equation*}
where 
\begin{equation*}
\tilde{T}^{\lambda/\rho^2}\tilde{g}(y, t) := \int_{\hat{\R}^n} e^{i \tilde{\phi}^{\lambda/\rho^2}(y,t;\eta)} \tilde{a}^{\lambda}(z;\eta) \tilde{g}(\eta)\,\ud \eta
\end{equation*}
for the phase $\tilde{\phi}(y,t;\eta)$ given by
\begin{equation}\label{parabolic 1}
\langle y, \eta \rangle +  \int_0^1 (1-r)\langle \partial_{\xi'\xi'}^2 \phi(\Upsilon_{\omega}(D_{\rho^{-1}}'y,t);\eta_n\omega + r\rho^{-1}\eta', \eta_n)\eta', \eta'  \rangle\,\ud r
\end{equation}
and the amplitude $\tilde{a}(y,t; \eta):= a(\Upsilon_{\omega}(D_{\rho^{-1}}'y; t); \eta_n \omega + \rho^{-1} \eta', \eta_n)$. In particular, by a change of spatial variables, it follows that
\begin{equation*}
\|T^{\lambda}g\|_{L^p(B_R)} \lesssim_{\phi} \rho^{(n+1)/p} \|\tilde{T}^{\lambda/\rho^2}\tilde{g}\|_{L^p( (\Upsilon_{\omega}^{\lambda} \circ D_{\rho})^{-1}(B_R))}
\end{equation*}

Fix a collection $\mathcal{B}_{R/\rho^2}$ of finitely-overlapping $R/\rho^2$-balls which cover $(\Upsilon_{\omega}^{\lambda} \circ D_{\rho})^{-1}(B_R)$ and observe that
\begin{equation*}
\|T^{\lambda}g\|_{L^p(B_R)} \lesssim_{\phi} \rho^{(n+1)/p}  \big(\sum_{B_{R/\rho^2} \in \mathcal{B}_{R/\rho^2}} \|\tilde{T}^{\lambda/\rho^2}\tilde{g}\|_{L^p(B_{R/\rho^2})}^p \big)^{1/p}.
\end{equation*}
It will be shown below that
\begin{align}
\nonumber
\|\tilde{T}^{\lambda/\rho^2}\tilde{g}\|_{L^p(B_{R/\rho^2})} \lesssim_{\varepsilon, \phi}\; & \mathfrak{D}_{\mathbf{1}}^{\varepsilon}(\lambda/\bar{C}\rho^2; R/\bar{C}\rho^2)(R/\rho^2)^{\alpha(p) + \varepsilon} \|\tilde{T}^{\lambda/\rho^2}\tilde{g}\|_{L^{p,R/\rho^2}_{\mathrm{dec}}(w_{B_{R/\rho^2}})} \\
\label{parabolic 2}
& \qquad  + (R/\rho^2)^{2n}(\lambda/R)^{-\varepsilon N/8n}\|g\|_{L^2(\hat{\R}^n)}
\end{align}
holds for each  $B_{R/\rho^2} \in \mathcal{B}_{R/\rho^2}$ and a suitable constant $\bar C \geq 1$, depending on $\phi$. Momentarily assuming this (which would follow immediately from the definitions if $(\tilde{\phi}, \tilde{a})$ were a type $\mathbf{1}$ datum), the proof of Lemma~\ref{parabolic rescaling lemma} may be completed as follows.

Since $\Upsilon_{\omega}$ is a diffeomorphism, it follows that
\begin{equation*}
\bigcup_{B_{R/\rho^2} \in \mathcal{B}_{R/\rho^2}} B_{R/\rho^2} \subseteq (\Upsilon_{\omega}^{\lambda} \circ D_{\rho})^{-1}(B_{C_{\phi}R})
\end{equation*}
where $B_{C_{\phi}R}$ is the ball concentric to $B_R$ but with radius $C_{\phi}R$ for some suitable choice of constant $C_{\phi} \geq 1$ depending on $\phi$. Thus, one may sum the $p$th power of both sides of \eqref{parabolic 2} over all the balls in $\mathcal{B}_{R/\rho^2}$ and reverse the changes of variables (both in spatial and frequency) to conclude that\footnote{Here one picks up $O(\rho^{n+1})$ copies of the error term $(R/\rho^2)^{2n}(\lambda/R)^{-N/8}\|g\|_{L^2(\hat{\R}^n)}$ from \eqref{parabolic 2}: that is, one for each ball in the collection $\mathcal{B}_{R/\rho^2}$. This is compensated for by the factor $\rho^{-4n}$ appearing in each of these errors; it is for this reason that the $R^{2n}$ factor is included in the definition of $\mathfrak{D}_{\mathbf{A}}^{\varepsilon}(\lambda; R)$ in \eqref{decoupling constant definition}.}
\begin{align*}
\|T^{\lambda}g\|_{L^p(B_R)} \lesssim_{\varepsilon, \phi, N}\; & \mathfrak{D}_{\mathbf{1}}^{\varepsilon}(\lambda/\bar{C}\rho^2; R/\bar{C}\rho^2)(R/\rho^2)^{\alpha(p) + \varepsilon} \big( \!\!\!\!\!\!\!\!\!\! \sum_{\tilde{\theta} : (R/\rho^2)^{-1/2}-\mathrm{plate}} \!\!\!\!\!\!\!\!\!\! \|T^{\lambda}g_{\theta}\|_{L^p(w_{B_{R}})}^p \big)^{1/p}\\
& \qquad + R^{2n}(\lambda/R)^{-\varepsilon N/8n}\|g\|_{L^2(\hat{\R}^n)},
\end{align*}
where $\theta$ is the image of $\tilde{\theta}$ under the map $(\eta', \eta_n) \mapsto (\rho(\eta' - \eta_n\omega), \eta_n)$. In particular, if $\omega_{\tilde{\theta}}$ denotes the centre of the $(R/\rho^2)^{-1/2}$-plate $\tilde{\theta}$, then 
\begin{equation*}
\theta = \big\{ (\xi',\xi_n) \in \hat{\R}^n : 1/2 \leq \xi_n \leq 2 \textrm{ and } |\omega + \rho^{-1} \omega_{\tilde{\theta}} - \xi'/\xi_n| < R^{-1/2} \big\},
\end{equation*}
and so the $\theta$ form a cover of the support of $g$ by $R^{-1/2}$-plates. This establishes the desired inequality \eqref{parabolic rescaling inequality} with a sharp cut-off appearing in the left-hand norm, rather than the weight function $w_{B_R}$. The strengthened result, with the weight, easily follows by pointwise dominating $w_{B_R}$ by a suitable rapidly decreasing sum of characteristic functions of $R$-balls.

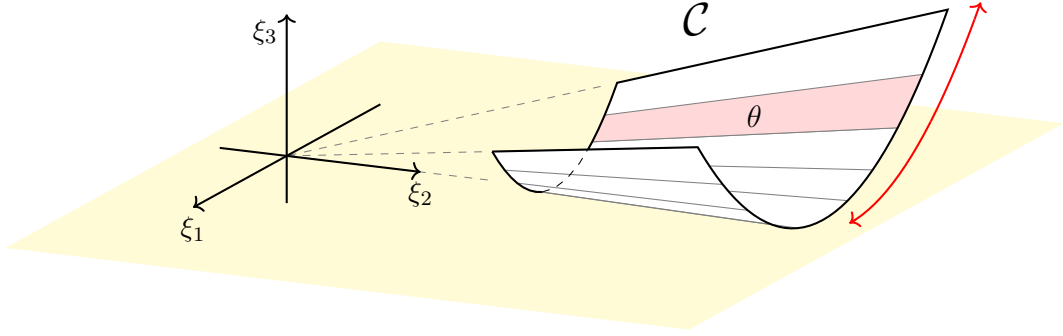
\begin{figure}
\begin{center}
\tdplotsetmaincoords{75}{25}
\resizebox {1.1\textwidth} {!} {\begin{tikzpicture}[tdplot_main_coords, scale=4]

	\begin{scope}

% xy plane

     \fill[fill=yellow!20] (-0.35,-1.5, 0) --  (-0.35, 1.5, 0) --  (2.2, 1.5, 0) --  (2.2,-1.5, 0) -- cycle;

% ghost cone

\draw[gray, thin, dashed] 
		({1*(19/20)},{(-1/2)*(19/20)},{(1/4)*(19/20)}) --({0},{0},{0})
    ;
\draw[gray, thin, dashed] 
		({1*(19/20)},{(1/2)*(19/20)},{(1/4)*(19/20)}) --({0},{0},{0})
   ;
\draw[gray, thin, dashed] 
		({1*(3/4)},{0},{0}) --({0},{0},{0})
    ; 

% co-ordinate axes

   \draw[black, thick, ->] 
(-1/4,0,0) -- (1/2,0,0)  
    ;
   \draw[black, thick, <-] 
(0,-3/4,0) -- (0,3/4,0)  
    ;
   \draw[black, thick, ->] 
(0,0,-1/6) -- (0,0,1/2)  
    ;

% fill back of cone

     \fill[fill=white] 
     ({2},{1}, {1/2})
    \foreach \t in {0.95, 0.9,...,0}
    {
        -- ({2},{\t},{(\t*\t)/2})
    }
        -- ({2},{0},{0}) -- ({1}, {0}, {0})
 \foreach \t in {0,0.05,...,0.55}
    {
        --({1},{\t},{(\t*\t)})
    }
       -- cycle;
    \foreach \a in {0.1,0.2,0.3}
    {
        \draw[gray, thin] 
		({1},{\a},{\a*\a}) --({2},{2*\a},{2*(\a*\a)});
    }
% fill front of cone

     \fill[fill=white] 
     ({2},{-0.2}, {(0.2)*(0.2)/2})
    \foreach \t in {-0.15, -0.1,...,0}
    {
        -- ({2},{\t},{(\t*\t)/2})
    }
      -- ({2},{0},{0}) -- ({1}, {0}, {0})
 \foreach \t in {-0.05,-0.1}
    {
        --({1},{\t},{(\t*\t)})
    }
       -- cycle;

     \fill[fill=white] 
     ({2},{-1}, {1/2})
    \foreach \t in {-0.95, -0.9,...,0}
    {
        -- ({2},{\t},{(\t*\t)/2})
    }
      -- ({2},{0},{0}) -- ({1}, {0}, {0})
 \foreach \t in {-0.05,-0.1,...,-0.5}
    {
        --({1},{\t},{(\t*\t)})
    }
      --({1},{-0.5},{(0.5*0.5)}) -- cycle;
 
 \foreach \a in {-0.1,-0.2,...,-0.5}
 {
    \draw[gray, thin] 
	({1},{\a},{\a*\a}) --({2},{2*\a},{2*(\a*\a)});
 }

% fill in plate

     \fill[fill=red!15] 
     ({2},{0.6}, {(0.6)*(0.6)/2})
    \foreach \t in {0.65, 0.7,...,0.8}
    {
        -- ({2},{\t},{(\t*\t)/2})
    }
       -- ({2}, {0.8}, {(0.8)*(0.8)/2}) -- ({1}, {0.4}, {(0.4)*(0.4)})
 \foreach \t in {0.39,0.38,...,0.3}
    {
        --({1},{\t},{(\t*\t)})
    }
      --({1},{0.3},{(0.3*0.3)})  -- cycle;
    \foreach \a in {0.3,0.4}
    {
        \draw[gray, thin] 
		({1},{\a},{\a*\a}) --({2},{2*\a},{2*(\a*\a)});
    }

% perimeter of the cone (foreground)

   \draw[black, thick] 
		({1},{-0.1},{(0.1)*(0.1)})
    \foreach \t in {-0.1,-0.15,...,-0.5}
    {
        --({1},{\t},{(\t*\t)})
    };
   \draw[black, dashed] 
		({1},{-0.1},{(0.1)*(0.1)})
    \foreach \t in {0,0.1,...,0.4}
    {
        --({1},{\t},{(\t*\t)})
    };

 \draw[black, thick] 
      ({1},{0.27},{(0.27)*(0.27)})
  \foreach \t in {0.3,0.35,...,0.5}
    {
        --({1},{\t},{(\t*\t)})
    }
  -- ({1},{1/2},{1/4}) -- ({2},{1},{1/2})
  \foreach \t in {1.0,0.95,...,-1}
    {
        --({2},{\t},{(\t*\t)/2})
    }
      --({2},{-1},{1/2}) -- (1, -1/2, 1/4)
    ;

% red arrows

  \draw[red, thick] 
		({21/10},{0.65},{(10/21)*(0.65)*(0.65)})
  \foreach \t in {0.65,0.7,...,1}
    {
        --({21/10},{\t},{10*(\t*\t)/21})
    }
       --  ({21/10},{1},{10/21}) ;
  \draw[red, thick, ->] 
       ({21/10},{1},{10/21})  --  ({21/10},{1.05},{(10/21)*(1.05)*(1.05)})
    ;
    \draw[red, thick] 
		({21/10},{0.65},{(10/21)*(0.65)*(0.65)})
  \foreach \t in {0.65,0.6,...,0}
    {
        --({21/10},{\t},{10*(\t*\t)/21})
    }
       --  ({21/10},{0},{0}) ;
 \draw[red, thick, <-] 
    ({21/10},{0},{0})  --  ({21/10},{0.05},{(10/21)*(0.05)*(0.05)})
    ;

% labels

\node at  ({1.5}, {(1.5)*(0.35)}, {(1.5)*(0.35)*(0.35)}) {\Large$\theta$};
\node[above,scale=1.5] at  ({5/4}, {3/5}, {3/8}) {\Large$\mathcal{C}$};

\node[below] at  ({0}, {-3/4}, {0}) {\large $\xi_1$};		
\node[below] at  ({1/2}, {0}, {0}) {\large $\xi_2$};	
\node[left] at  ({0}, {0}, {7/16}) {\large $\xi_3$};

		\end{scope}
			\end{tikzpicture}}
\caption{The simplest case of the parabolic rescaling lemma, corresponding to the phase $\phi(x, t; \xi) := x_1\xi_1 + x_2 \xi_2 + t \xi_1^2/\xi_2$. Here each plate is associated with a subset of the conic surface $\mathcal{C}$ defined by $\xi_3 = \xi_1^2/\xi_2$ for $1/2 \leq \xi_2 \leq 2$. The key observation is that there exists an affine transformation of the ambient space which essentially maps $\theta$ to the whole of $\mathcal{C}$.}

\end{center}
\end{figure}

It remains to show the validity of the inequality \eqref{parabolic 2} for each $B_{R/\rho^2} \in \mathcal{B}_{R/\rho^2}$. Let $\mathrm{L} \in \mathrm{GL}(n,\R)$ be such that $\mathrm{L}e_n=e_n$ and
\begin{equation}\label{law of inertia}
\partial_{\eta' \eta'}^2 \partial_t \tilde{\phi}_\mathrm{L} (0,0;e_n)=I_{n-1,\sigma_+}
\end{equation}
for some $0 \leq \sigma_+ \leq n-1$, where
$$
\tilde{\phi}_\mathrm{L}(y,t;\eta):=\tilde{\phi}(\mathrm{L}^{-1} y, t ; \mathrm{L}\eta).
$$
Observe that $\mathrm{L}$ is a composition of a rotation and an anisotropic dilation given by the matrix $\mathrm{diag}(\sqrt{|\mu_1|}, \dots, \sqrt{|\mu_{n-1}|}, 1)$ where the $\mu_j$ are the eigenvalues of \eqref{matrix eigenvalues} evaluated at $(0,0;e_n)$. By a linear change of both the $y$ and $\eta$ variables, it suffices to show that \eqref{parabolic 2} holds with $\tilde{T}^{\lambda/\rho^2}\tilde{g}$ replaced with $\tilde{T}_{\mathrm{L}}^{\lambda/\rho^2}\tilde{g}_{\mathrm{L}}$ where $\tilde{T}_{\mathrm{L}}^{\lambda/\rho^2}$ is defined with respect to the datum $(\tilde{\phi}_{\mathrm{L}}, \tilde{a}_{\mathrm{L}})$ for $\tilde{\phi}_\mathrm{L}$ as above, $\tilde{a}_\mathrm{L}(y,t;\eta):=\tilde{a}(\mathrm{L}^{-1}y, t ; \mathrm{L} \eta)$ and $\tilde{g}_\mathrm{L} := |\det \mathrm{L}| \cdot \tilde{g} \circ \mathrm{L}$. This would follow from the definition of $\mathfrak{D}_{\mathbf{1}}^{\varepsilon}(\lambda; R)$ and Lemma~\ref{basic rescaling lemma} provided that the new datum $(\tilde{\phi}_\mathrm{L}, \tilde{a}_\mathrm{L})$ is of type $(1,1,\bar{C})$ for some suitable choice of constant $\bar{C} \geq 1$. It is remarked that the amplitude $\tilde{a}_\mathrm{L}$ may not satisfy the required support conditions described at the beginning of $\S$\ref{subsection: properties of the phase}; however, by decomposing the operator, as in the proof of Lemma~\ref{basic rescaling lemma}, this issue may easily be resolved. On the other hand, if $\bar{C}$ is suitably chosen, it is clear that $\tilde{a}_L$ satisfies the required margin condition. 

To verify the remaining hypotheses in the definition of type $(1,1,\bar{C})$ data, first note that, by retracing the steps of the argument prior to \eqref{parabolic 1}, one deduces that
\begin{equation}\label{verifying geometric conditions 0}
\tilde{\phi}_\mathrm{L}(y,t;\eta) = \rho^2 \phi( \Upsilon_{\omega}( D_{\rho^{-1}}' \circ \mathrm{L}^{-1} y,t), t; \eta_n\omega + \rho^{-1} \mathrm{L}'\eta', \eta_n).
\end{equation}
Alternatively, using \eqref{parabolic 1} directly, $\tilde{\phi}_\mathrm{L}(y,t; \eta)$ can be expressed as
\begin{equation}\label{taylor expanded tilde B}
\langle y, \eta \rangle +  \int_0^1 (1-r)\langle \partial_{\xi'\xi'}^2 \phi(\Upsilon_{\omega}( D_{\rho^{-1}}' \circ \mathrm{L}^{-1} y,t);\eta_n\omega + r\rho^{-1} \mathrm{L}'\eta', \eta_n) \mathrm{L}' \eta', \mathrm{L}' \eta'  \rangle\,\ud r,
\end{equation}
where $\mathrm{L}'$ is the top-left $(n-1) \times (n-1)$ submatrix of $\mathrm{L}$. These two formul\ae\, are used in conjunction to yield bounds on various derivatives of $\tilde{\phi}_{\mathrm{L}}$. To this end, it is also useful to note that, by definition of $\Upsilon$ and the inverse function theorem,
\begin{equation*}
\partial_{y}\Upsilon(y,t; \omega, 1) = \partial_{\xi x}^2 \phi (\Upsilon_{\omega}(y,t); \omega, 1))^{-1}, 
\end{equation*}
so each entry $\partial_{y_j} \Upsilon^i(y,t;\omega,1)$ of the above matrix may be written as the product of $[\det(\partial_{ \xi x}^2 \phi (\Upsilon_{\omega}(y,t); \omega, 1))]^{-1}$ and a polynomial expression in $(\partial_{ \xi_l} \partial_{x_k} \phi)(\Upsilon_{\omega}(y,t); \omega, 1)$.

First consider the technical conditions on the derivatives. Differentiating the formula \eqref{verifying geometric conditions 0} and assuming $\rho$ is sufficiently large, depending on $\phi$, immediately implies that $(\tilde{\phi}_\mathrm{L}, \tilde{a}_\mathrm{L})$ satisfies conditions D1$_{\mathbf{1}}$) and D2$_{\mathbf{1}}$) for $|\beta'| \geq 2$. The remaining cases of D1$_{\mathbf{1}}$) and D2$_{\mathbf{1}}$) can then be readily deduced by differentiating \eqref{taylor expanded tilde B}.

Concerning H1$_{\mathbf{1}}$), by differentiating \eqref{taylor expanded tilde B} and using the conditions D1$_{\mathbf{A}}$) and D2$_{\mathbf{A}}$) of $(\phi, a)$, one deduces that
\begin{equation*}
\partial_{\eta y}^2 \tilde{\phi}_\mathrm{L} (y,t; \eta)= I_n + O_{\phi}(\rho^{-1}).
\end{equation*}
Thus, H1$_{\mathbf{1}}$) holds for $(\tilde{\phi}_\mathrm{L}, \tilde{a}_\mathrm{L})$ provided $\rho$ is sufficiently large depending on $\phi$. Note that the conditions D1$_{\mathbf{A}}$) and D2$_{\mathbf{A}}$) are used here so as to ensure the dependence on $\phi$ is as described in Remark~\ref{parabolic rescaling remark}.

Concerning H2$_{\mathbf{1}})$, the homogeneity of $\phi$ and \eqref{law of inertia} imply that
\begin{equation*}
\partial_{\eta' \eta'}^2 \partial_t \tilde{\phi}_\mathrm{L} (z; \eta) - \frac{1}{\eta_n} I_{n-1, \sigma_+} = \frac{1}{\eta_n} \big(\partial_{\eta' \eta'}^2 \partial_t \tilde{\phi}_\mathrm{L} (z; \eta'/\eta_n, 1) - \partial_{\eta' \eta'}^2 \partial_t \tilde{\phi}_\mathrm{L} (0; e_n)\big).
\end{equation*}
In particular, for $1 \leq i, j \leq n-1$, the $(i,j)$ entry of the above matrix equals 
\begin{equation*}
 \int_0^1 \langle \partial_{\eta'} \partial_{\eta_i \eta_j}^2 \partial_t \tilde{\phi}_\mathrm{L} (rz; r\eta'/\eta_n, 1), \eta'/\eta_n \rangle + \langle \partial_z \partial^2_{\eta_i \eta_j} \partial_t \tilde{\phi}_\mathrm{L}(rz; r\eta'/\eta_n, 1), z  \rangle \ud r .
\end{equation*}
Since it has been shown above that the datum $(\tilde{\phi}_\mathrm{L}, \tilde{a}_\mathrm{L})$ satisfies D1$_{\mathbf{1}}$) and D2$_{\mathbf{1}})$, the integrand in the above expression may now be bounded above in absolute value by $c_{\mathrm{par}}$. Thus, $(\tilde{\phi}_\mathrm{L}, \tilde{a}_\mathrm{L})$ also satisfies H2$_{\mathbf{1}}$) follows and therefore is of type $(1,1,\bar C)$, as required.
 
\end{proof}

%%%%%%%%%%%%%%%%%%%%%%%%%%%%%%%%%%%%%%%%%%%%%%%%%%%%%%%%%%%%%%%%%%%%%%%%%%%%%%%%%%%%%%%%%%%%%%%%

%                                         APPROXIMATION BY EXTENSION OPERATORS

%%%%%%%%%%%%%%%%%%%%%%%%%%%%%%%%%%%%%%%%%%%%%%%%%%%%%%%%%%%%%%%%%%%%%%%%%%%%%%%%%%%%%%%%%%%%%%%%

\subsection{Approximation by extension operators} This subsection deals with an approximation lemma which allows one to use Theorem~\ref{translation-invariant decoupling theorem} to bound variable coefficient operators at small spatial scales.

Let $T^{\lambda}$ be an operator associated to a type $\mathbf{1}$ datum $(\phi, a)$. For each $\bar{z} \in \R^{n+1}$ with $\bar{z}/\lambda \in Z$ the map $\eta \mapsto (\partial_z\phi^{\lambda})(\bar{z}; \Psi^{\lambda}(\bar{z}; \eta))$ is a graph parametrisation of a hypersurface $\Sigma_{\bar{z}}$ with precisely one vanishing principal curvature at each point. In particular, recalling \eqref{Psi}, one has
\begin{equation*}
\langle z, (\partial_z\phi^{\lambda})(\bar{z}; \Psi^{\lambda}(\bar{z}; \eta)) \rangle = \langle x, \eta  \rangle + t h_{\bar{z}}(\eta) \qquad \textrm{for all $z = (x,t) \in \R^{n+1}$}
\end{equation*}
where $h_{\bar{z}}(\eta) := (\partial_t\phi^{\lambda})(\bar{z}; \Psi^{\lambda}(\bar{z}; \eta))$. Let $E_{\bar{z}}$ denote the extension operator associated to $\Sigma_{\bar{z}}$, given by
\begin{equation*}
E_{\bar{z}}g(x,t) := \int_{\hat{\R}^n} e^{ i (\langle x, \eta  \rangle + th_{\bar{z}}(\eta))} a_{\bar{z}}(\eta) g(\eta)  \,\ud \eta \qquad \textrm{for all $(x,t) \in \R^{n+1}$}
\end{equation*}
where $a_{\bar{z}}(\eta) := a_2\circ\Psi^{\lambda}(\bar{z}; \eta)|\det \partial_{\eta} \Psi^{\lambda}(\bar{z}; \eta)|$. The operator $T^\lambda$ is effectively approximated by $E_{\bar z}$ at small spatial scales. Furthermore, the conditions on the translation-invariant decoupling inequality, Theorem~\ref{translation-invariant decoupling theorem}, are satisfied by each of the functions $h_{\bar{z}}$. In particular, the type $\mathbf{1}$ condition implies the following uniform bound.

\begin{lemma}\label{uniformity curvature extension}
Let $(\phi, a)$ be a type $\mathbf{1}$ datum. Each eigenvalue $\mu$ of $\partial_{\eta' \eta'} h_{\bar z}$ satisfies $|\mu| \sim 1$ on $\mathrm{supp}\,a_{\bar{z}}$.
\end{lemma}

The proof of this lemma is an elementary calculus exercise, the details of which are omitted. 

Concerning the approximation of $T^\lambda$ by $E_{\bar z}$, suppose that $1 \leq K \leq \lambda^{1/2}$ and $z \in B(\bar{z}, K) \subseteq B(0, 3\lambda/4)$; this containment property may be assumed in view of the margin condition M$_\mathbf{1}$). By applying the change of variables $\xi = \Psi^\lambda(\bar z; \eta)$ and a Taylor expansion of $\phi^\lambda$ around the point $\bar z$,
\begin{equation}\label{simple 1}
T^{\lambda}f(z) = \int_{\hat{\R}^n} e^{ i (\langle z-\bar{z}, (\partial_z\phi^{\lambda})(\bar{z}; \Psi^{\lambda}(\bar{z}; \eta)) \rangle+\mathcal{E}^{\lambda}_{\bar{z}}(z-\bar{z}; \eta))}  a_1^{\lambda}(z)a_{\bar{z}}(\eta)f_{\bar{z}}(\eta)\,\ud \eta 
\end{equation}
where $f_{\bar{z}} := e^{i \phi^{\lambda}(\bar{z}; \Psi^{\lambda}(\bar{z};\,\cdot\,))} f\circ \Psi^{\lambda}(\bar{z}; \,\cdot\,)$ and, by Taylor's theorem,
\begin{equation}\label{error 1}
\mathcal{E}^{\lambda}_{\bar{z}}(v; \eta) = \frac{1}{\lambda}\int_0^1 (1-r) \langle (\partial_{zz}^2\phi)((\bar{z}+rv)/\lambda; \Psi^{\lambda}(\bar{z};\eta))v, v \rangle\,\ud r.
\end{equation}
Since $K \leq \lambda^{1/2}$ and $(\phi, a)$ is type $\mathbf{1}$, so that property D2$_{\mathbf{1}}$) holds, it follows that 
\begin{equation*}
\sup_{(v;\eta) \in B(0,K) \times \mathrm{supp}\,a_{\bar{z}}}|\partial_{\xi}^{\beta} \mathcal{E}^{\lambda}_{\bar{z}}(v;\eta)| \lesssim_{N} 1
\end{equation*}
for all $\beta \in \N_0^n$ with $1 \leq |\beta| \leq 4N$.  Consequently, $\mathcal{E}^{\lambda}_{\bar{z}}(z-\bar{z}; \xi)$ does not contribute significantly to the oscillation induced by the exponential in \eqref{simple 1} and it can therefore be safely removed from the phase, at the expense of some negligible error terms.

\begin{lemma}\label{approximation lemma}
Let $T^\lambda$ be an operator associated to a type $\mathbf{1}$ datum $(\phi, a)$. Let $0 < \delta \leq 1/2$, $1 \leq K \leq \lambda^{1/2 - \delta}$ and $\bar{z}/\lambda \in Z$ so that $B(\bar z, K) \subseteq B(0, 3\lambda/4)$. Then
\begin{enumerate}[i)]
    \item \begin{equation}\label{T dominated by E}
\| T^\lambda f \|_{L^p(w_{B(\bar{z},K)})} \lesssim_{N} \| E_{\bar{z}} f_{\bar{z}} \|_{L^p(w_{B(0,K)})} + \lambda^{-\delta N/2 }  \| f\|_{L^2(\hat{\R}^n)}
\end{equation}
holds provided $N$ is sufficiently large depending on $n$, $\delta$ and $p$.
    \item Suppose that $|\bar{z}| \leq \lambda^{1- \delta'}$. There exists a family of operators $\mathbf{T}^{\lambda}$ all with phase function $\phi$ and assoicated to type $(1,1,C)$ data such that 
    \begin{equation}\label{E dominated by T}
\| E_{\bar{z}} f_{\bar{z}} \|_{L^p(w_{B(0,K)})} \lesssim_{N} \| T_*^\lambda f \|_{L^p(w_{B(\bar{z},K)})}  + \lambda^{-\min\{\delta, \delta' \} N/2 } \| f\|_{L^2(\hat{\R}^n)}
\end{equation}
holds for some $T_*^{\lambda} \in \mathbf{T}^{\lambda}$ provided $N$ is sufficiently large depending on $n$, $\delta$ and $p$. Moreover, the family $\mathbf{T}^{\lambda}$ has cardinality $O_N(1)$ and is independent of the choice of ball $B(\bar{z}, K)$.
\end{enumerate}
\end{lemma}

\begin{remark}
\begin{enumerate}[i)]
 \item The weights appearing in Lemma~\ref{approximation lemma} are defined with respect to the same $N \in \N$ as that appearing in the $\lambda$ exponent. This is also understood to be the same $N$ as that appearing in the definition of the D2$_{\mathbf{A}}$) condition. 
 \item If one replaces $w_{B(\bar{z},K)}$ with the characteristic function $\chi_{B(\bar{z},K)}$ on the left-hand side of \eqref{T dominated by E}, then the proof of Lemma~\ref{approximation lemma} shows that the inequality holds without the additional $\lambda^{-\delta N/2 }  \| f\|_{L^2(\hat{\R}^n)}$ term.
\end{enumerate}

\end{remark}

Several variants of this kind of approximation (or stability) lemma have previously appeared in the literature: see, for instance, \cite[Chapter VI, $\S$2]{Stein1993} or \cite{Christ1988, Tao1999}. In the context of decoupling, Lemma~\ref{approximation lemma} is closely related to certain approximation arguments used to extend decoupling estimates to wider classes of surfaces in \cite{Pramanik2007, Garrigos2010, Guo} and \cite[$\S\S$7-8]{Bourgain2015}. A variant of Lemma~\ref{approximation lemma} (which is somewhat cleaner than the above statement) can also be applied to slightly simplify the original proof of the decoupling theorem in \cite{Bourgain2015, Bourgain2017} and, in particular, obviate the need to reformulate the problem in terms of functions with certain Fourier support conditions (the details of the original `reformulation' approach are given in \cite[$\S$5]{Bourgain2017}).

\begin{proof}[Proof (of Lemma~\ref{approximation lemma})] Note that $f$ in \eqref{simple 1} may be replaced by $f \psi$ where $\psi$ is a smooth function that equals 1 on $\mathrm{supp}\,a_{\bar{z}}$ and vanishes outside its double. Moreover, recalling the definition of $a_{\bar{z}}$ and that $(\phi, a)$ is a type $\mathbf{1}$ datum, one may assume that the function $\psi$ is supported in $[0,2\pi]^n$. In view of the expression \eqref{simple 1}, by performing a Fourier series decomposition of $e^{ i \mathcal{E}_{\bar{z}}^\lambda (v,\eta)} \psi(\eta)$ in the $\eta$ variable, one may write
\begin{equation}\label{Fourier series}
e^{ i\mathcal{E}_{\bar{z}}^\lambda (v; \eta)} \psi(\eta) = \sum_{ \ell \in \Z^{n}} b_{\ell}(v) e^{ i \langle\ell , \eta \rangle}
\end{equation}
where
\begin{equation*}
b_{\ell}(v)= \int_{[0,2\pi]^n} e^{- i  \langle\ell,  \eta \rangle } e^{ i \mathcal{E}_{\bar{z}}^\lambda (v; \eta)} \psi(\eta) \ud \eta.
\end{equation*}
The formula \eqref{error 1} and property D2$_\mathbf{1}$) of the phase together imply that
\begin{equation*}
\sup_{\eta \in  [0,2\pi]^n}|\partial_{\eta}^{\beta} \mathcal{E}^{\lambda}_{\bar{z}}(v;\eta)| \lesssim_N \frac{|v|^2}{\lambda} 
\end{equation*}
for all multi-indices $\beta \in \N$ with $1 \leq |\beta| \leq N$. Therefore, by repeated application of integration-by-parts, one deduces that
\begin{equation*}
|b_{\ell}(v)| \lesssim_{N} (1 + |\ell|)^{-N} \quad \textrm{whenever $|v| \leq 2\lambda^{1/2}$}.
\end{equation*}
This, \eqref{Fourier series} and \eqref{simple 1} lead to the useful pointwise estimate
\begin{equation}\label{approximation 1}
|T^\lambda f(\bar{z}+v)| \lesssim_N \sum_{\ell \in \Z^{n}}  (1 + |\ell|)^{-N}|E_{\bar{z}} (f_{\bar{z}} e^{i\langle \ell, \,\cdot\,\rangle})(v)|,
\end{equation}
valid for $|v| \leq 2\lambda^{1/2}$. Writing 
\begin{equation*}
\|T^{\lambda}f\|_{L^p(w_{B(\bar{z}, K)})} \leq \|(T^{\lambda}f)\chi_{B(\bar{z},2\lambda^{1/2})}\|_{L^p(w_{B(\bar{z}, K)})} + \|(T^{\lambda}f)\chi_{\R^n\setminus B(\bar{z},2\lambda^{1/2})}\|_{L^p(w_{B(\bar{z}, K)})},
\end{equation*}
it follows from \eqref{approximation 1} that  
\begin{equation}\label{approximation 2}
\|(T^{\lambda}f)\chi_{B(\bar{z},2\lambda^{1/2})}\|_{L^p(w_{B(\bar z, K)})} \lesssim_{N} \sum_{\ell \in \Z^{n}}  (1 + |\ell|)^{-N} \|E_{\bar{z}} (f_{\bar{z}} e^{i\langle \ell, \,\cdot\,\rangle})\|_{L^p(w_{B(0, K)})}.
\end{equation}
On the other hand, it is claimed that 
\begin{equation}\label{approximation 3}
\|(T^{\lambda}f)\chi_{\R^n\setminus B(\bar{z},2\lambda^{1/2})}\|_{L^p(w_{B(\bar{z}, K)})} \lesssim \lambda^{n/2p - \delta(N-n+2) }  \|f\|_{L^2(\hat{\R}^n)}
\end{equation}
and therefore this term can be treated as an error. Indeed, if $|v|>2 \lambda^{1/2}$ and $K \leq \lambda^{1/2-\delta}$, then
\begin{equation*}
(1 + K^{-1} |v|)^{-(N-n+2)} \leq (1 + 2\lambda^{1/2}K^{-1})^{-(N-n+2)} \leq \lambda^{-\delta(N-n+2)}.
\end{equation*}
Combining this observation with the trivial estimate
\begin{equation*}
\| T^{\lambda} f \|_{L^p(\tilde{w}_{B(\bar{z},K)})} \lesssim K^{n/p}  \| f \|_{L^2(\hat{\R}^n)}
\end{equation*}
where $\tilde{w}_{B(0,K)}:=(1+K^{-1}|\,\cdot\,|)^{-(n+2)}$,  one readily deduces \eqref{approximation 3}.

Observe that the operator $E_{\bar{z}}$ enjoys the translation-invariance property 
\begin{equation}\label{approximation 4}
E_{\bar{z}}[e^{i \langle \ell, \,\cdot\,\rangle} g](x,t) = E_{\bar{z}}g(x+\ell,t) \qquad \textrm{for all $(x,t) \in \R^{n+1}$ and all $\ell \in \R^n$;}
\end{equation}
it is for this reason that the graph parametrisation was introduced at the outset of the argument. The identity \eqref{approximation 4} together with \eqref{approximation 2} and \eqref{approximation 3} imply that 
\begin{equation}\label{altogether}
\| T^\lambda f  \|_{L^p(w_{B(\bar{z},K)})} \lesssim_N \sum_{\ell \in \Z^{n}} (1 + |\ell|)^{-N}  \| E_{\bar{z}} f_{\bar{z}}  \|_{L^p(w_{B((\ell, 0),K)})} + \lambda^{-\delta N/2 }  \| f\|_{L^2(\hat{\R}^n)},
\end{equation}
provided $N$ is chosen to be suitably large. One may readily verify that
\begin{equation}\label{dominate weight}
\sum_{\ell \in \Z^{n}} (1+|\ell|)^{-N} w_{B((\ell, 0),K)} \lesssim w_{B(0,K)}
\end{equation}
and combining this with \eqref{altogether} immediately yields the desired estimate \eqref{T dominated by E}.

The proof of \eqref{E dominated by T} is similar to that of \eqref{T dominated by E}, although a slight complication arises since, in contrast with $E_{\bar{z}}$, the variable coefficient operator $T^\lambda$ does not necessarily satisfy the translation-invariance property described in \eqref{approximation 4}. 

One may write
\begin{equation*}
E_{\bar{z}} f_{\bar{z}} (v) =  \int_{\hat{\R}^n} e^{i \phi^{\lambda}(\bar{z}+v; \Psi^{\lambda}(\bar{z},\eta))}   e^{- i \mathcal{E}^{\lambda}_{\bar{z}}(v; \eta)} a_{\bar{z}}(\eta)  f\circ \Psi^{\lambda}(\bar{z}; \eta)\,\ud \eta
\end{equation*}
and, by forming the Fourier series expansion of $e^{- i \mathcal{E}^{\lambda}_{\bar{z}}(v; \eta)}\psi(\eta)$ in $\eta$ and undoing the change of variables $\xi= \Psi^\lambda (\bar z ; \eta)$, thereby deduce that
\begin{equation*}
|E_{\bar{z}} f_{\bar{z}}(v)|  \lesssim_N \sum_{\ell \in \Z^{n}} (1 + |\ell|)^{-4N} |T^\lambda [e^{ i\langle  \ell , (\partial_z\phi^{\lambda})(\bar{z}\,\cdot\,)\rangle } f](\bar{z} + v)|
\end{equation*}
whenever $|v| \leq 2\lambda^{1/2}$. This pointwise bound is understood to hold modulo the choice of spatial cut-off $a_1$ appearing in the definition of $T^{\lambda}$. Taking $L^p(w_{B(\bar{z},K)})$ norms in $z$ and reasoning as in the proof of \eqref{T dominated by E}, one obtains
\begin{equation*}
\| E_{\bar{z}} f_{\bar{z}} \|_{L^p(w_{B(0,K)})} \lesssim_N  \sum_{ \ell \in \Z^{n}} (1 + |\ell|)^{-4N}  \| (T^{\lambda} \tilde{f}_{\ell})\chi_{B(\bar{z}, 2\lambda^{1/2})}  \|_{L^p(w_{B(\bar{z},K)})} + \lambda^{- \delta N/2}  \| f \|_{L^2(\hat{\R}^n)},
\end{equation*}
where $\tilde{f}_{\ell} := e^{ i\langle  \ell , (\partial_z\phi^{\lambda})(\bar{z};\,\cdot\,)\rangle } f$. Note that the cut-off function $\chi_{B(\bar{z}, 2\lambda^{1/2})}$ can be dominated by a smooth amplitude $\tilde{a}_1^{\lambda}$ where $\tilde{a}_1$ is equal to 1 on $\mathrm{supp}\,a_1$ and has half the margin. The above sum is split into a sum over $\ell$ satisfying $|\ell| > C_N$ and a sum over the remaining $\ell$ where $C_N$ is a constant depending on $N$, chosen large enough for the present purpose. To control sum over large $\ell$, apply \eqref{T dominated by E} and argue as in \eqref{dominate weight} to conclude that
\begin{align*}
\sum_{\substack{ \ell \in \Z^{n} \\[1pt] |\ell| > C_N}} (1 + |\ell|)^{-4N}  \| T^{\lambda} \tilde{f}_{\ell} \|_{L^p(w_{B(\bar{z},K)})} &\lesssim_N \sum_{\substack{ \ell \in \Z^{n} \\[1pt] |\ell| > C_N}}  (1 + |\ell|)^{-2N} \| E_{\bar{z}} f_{\bar{z}}   \|_{L^p(w_{B((\ell, 0),K)})} \\
& \lesssim C_N^{-N} \|E_{\bar{z}} f_{\bar{z}}\|_{L^p(w_{B(0,K)})}.
\end{align*}
Therefore, if $C_N$ is chosen to be sufficiently large depending on $N$, the above term can be absorbed into the left-hand side of the inequality and one obtains 
\begin{equation*}
\|E_{\bar{z}} f_{\bar{z}}\|_{L^p(w_{B(0,K)})} \lesssim_N \sum_{\substack{\ell \in \Z^n \\ |\ell| \leq C_N}}\| T^\lambda \tilde{f}_\ell  \|_{L^p(w_{B(\bar{z},K)})} + \lambda^{-\delta N/2 }  \| f\|_{L^2(\hat{\R}^n)}.
\end{equation*}

Each $T^\lambda \tilde{f}_\ell$ can be thought of as an operator $T^{\lambda}_\ell$ where the latter has phase $\phi$ and amplitude function
\begin{equation*}
    \tilde{a}_{\ell}(z;\xi) := \tilde{a}_1(z;\xi)e^{ i\langle  \ell , (\partial_z\phi^{\lambda})(\bar{z};\xi)\rangle } .
\end{equation*}
Unfortunately, these amplitudes depend on the choice of ball $B(\bar{z},K)$ and therefore are unsuitable for the present purpose. To remove this undesirable dependence, one may take a Taylor series expansion to write
\begin{equation}\label{long Taylor expansion}
    e^{ i\langle  \ell , (\partial_z\phi^{\lambda})(\bar{z};\xi)\rangle } = \sum_{|\alpha| \leq N-1} u_{\alpha}(\omega) \Big(\frac{\bar{z}}{\lambda} \Big)^{\alpha} + O\big((|\bar{z}|/\lambda)^{N}\big)
\end{equation}
where each $u_{\alpha} \in C^{\infty}(\R^n)$ satisfies $|\partial_{\xi}^{\beta} u_{\alpha}(\xi)| \lesssim_N 1$ for all $|\beta| \leq N$. Note that the $u_{\alpha}$ do not depend on the choice of $\bar{z}$. Furthermore, since $|\bar{z}| \leq \lambda^{1 - \delta'}$, it follows that the error in \eqref{long Taylor expansion} is $O(\lambda^{- \delta' N})$ and the part of the operator arising from such terms can be bounded by $\lambda^{- \min \{\delta, \delta' \} N/2 }  \| f\|_{L^2(\hat{\R}^n)}$. The family of operators $\mathbf{T}^{\lambda}$ is now given by the family of amplitudes
\begin{equation*}
    u_{\alpha}(\omega)\tilde{a}_{\ell}(z;\xi), \qquad |\ell| \leq C_n, |\alpha| \leq N-1.
\end{equation*}
Since $|\bar{z}|/\lambda \leq 1$, one concludes that 
\begin{equation*}
\|E_{\bar{z}} f_{\bar{z}}\|_{L^p(w_{B(0,K)})} \lesssim_N \sum_{T_*^{\lambda} \in \mathbf{T}^{\lambda}}\| T_*^\lambda \tilde{f}_\ell  \|_{L^p(w_{B(\bar{z},K)})} + \lambda^{-\min \{\delta, \delta' \} N/2 }  \| f\|_{L^2(\hat{\R}^n)}
\end{equation*}
and the desired inequality now holds for some choice of $T_*^{\lambda} \in \mathbf{T}^{\lambda}$ by pigeonholing.

\end{proof}

%%%%%%%%%%%%%%%%%%%%%%%%%%%%%%%%%%%%%%%%%%%%%%%%%%%%%%%%%%%%%%%%%%%%%%%%%%%%%%%%%%%%%%%%%%%%%%%%

%                     PROOF OF THE VARIABLE COEFFICIENT DECOUPLING ESTIMATE

%%%%%%%%%%%%%%%%%%%%%%%%%%%%%%%%%%%%%%%%%%%%%%%%%%%%%%%%%%%%%%%%%%%%%%%%%%%%%%%%%%%%%%%%%%%%%%%%

\subsection{Proof of the variable coefficient decoupling estimates}\label{subsec:proof}

By the discussion in $\S\S$\ref{subsection: properties of the phase}-\ref{subsection: parabolic rescaling}, to prove Theorem~\ref{decoupling theorem} for the fixed parameters $2 \varepsilon$, $M$ and $p$ it suffices to show 
\begin{equation*}
\mathfrak{D}^{\varepsilon}_{\mathbf{1}}(\lambda; R) \lesssim_{\varepsilon} 1 \qquad \textrm{for all $1 \leq R \leq \lambda^{1-\varepsilon/n}$.}
\end{equation*}
The trivial estimate \eqref{trivial decoupling} implies the above inequality if $R$ is small (that is, $R \lesssim_{\varepsilon} 1$), and the proof proceeds by induction on $R$, using this observation as the base case. In particular, one may assume by way of induction hypothesis that the following holds.

\begin{hypothesis} There is a constant $\bar{C}_{\varepsilon} \geq 1$, depending only on the admissible parameters $n$, $\varepsilon$, $M$ and $p$, such that
\begin{equation*}
\mathfrak{D}_{\mathbf{1}}^{\varepsilon}(\lambda';R') \leq \bar{C}_{\varepsilon} 
\end{equation*}
holds for all $1 \leq R' \leq R/2$ and all $\lambda'$ satisfying $R' \leq (\lambda')^{1 - \varepsilon/n}$.
\end{hypothesis}

Let $\mathcal{B}_K$ be a finitely-overlapping cover of $B_R$ by balls of radius $K$ for some $2 \leq K \leq \lambda^{1/4}$. One may assume that any centre $\bar{z}$ of a ball in this cover satisfies $|\bar{z}| \leq \lambda^{1- \varepsilon/n}$. The estimate \eqref{T dominated by E} from Lemma~\ref{approximation lemma} with $\delta = 1/4$ implies that 
\begin{equation*}
\|T^{\lambda}f\|_{L^p(B_R)} \lesssim \big(\sum_{B(\bar{z},K) \in \mathcal{B}_K} \|E_{\bar{z}} f_{\bar{z}}\|_{L^p(w_{B(0,K)})}^p \big)^{1/p} + R^{n+1}(\lambda/R)^{-N/8}\|f\|_{L^2(\hat{\R}^n)}.
\end{equation*}
Applying the theorem of Bourgain--Demeter \cite{Bourgain2015, Bourgain2017a} (that is, Theorem~\ref{translation-invariant decoupling theorem}) with exponent $\varepsilon/2$ (and recalling Lemma~\ref{uniformity curvature extension}), one deduces that the inequality
\begin{equation*}
\|E_{\bar{z}} f_{\bar{z}}\|_{L^p(w_{B(0,K)})} \lesssim_{\varepsilon} K^{\alpha(p) + \varepsilon/2}\|E_{\bar{z}} f_{\bar{z}}\|_{L^{p, K}_{\mathrm{dec}}(w_{B(0,K)})}
\end{equation*}
holds for each of the extension operators in the previous display. Combining these observations with an application of \eqref{E dominated by T} from Lemma~\ref{approximation lemma} with $\delta'=\varepsilon/n$, and summing over $\mathcal{B}_K$,
\begin{equation*}
\|T^{\lambda}f\|_{L^p(B_R)} \lesssim_{\varepsilon} K^{\alpha(p) + \varepsilon/2} \big(\!\!\!\!\sum_{\sigma : K^{-1/2}-\mathrm{plate}} \!\!\!\!\|T^{\lambda}f_{\sigma}\|_{L^p(w_{B_R})}^p\big)^{1/p} +  R^{n+1} (\lambda/R)^{-\varepsilon N/8n} \|f\|_{L^2(\hat{\R}^n)}.
\end{equation*}
The operator on the right involves a slightly different amplitude function compared with that on the left but, as in the statement of Theorem \ref{decoupling theorem}, this is suppressed in the notation.

Note that, since $K \geq 2$, $\bar{C} \geq 1$ and $R \leq \lambda^{1 - \varepsilon/n}$, trivially $R/\bar{C}K \leq (\lambda/\bar{C}K)^{1-\varepsilon/n}$ and $R/\bar{C}K \leq R/2$. Consequently, the assumptions of the radial induction hypothesis are valid for  the parameters $R' := R/\bar{C}K$ and $\lambda' := \lambda/\bar{C}K$. Thus, by combining the radial induction hypothesis with \eqref{parabolic rescaling inequality} from the parabolic rescaling lemma, one concludes that
\begin{equation*}
\|T^{\lambda}f\|_{L^p(B_R)} \leq C_{\varepsilon} \bar{C}_{\varepsilon} K^{-\varepsilon/2} R^{\alpha(p)+\varepsilon} \|T^{\lambda}f\|_{L^{p, R}_{\mathrm{dec}}(w_{B_{R}})} + R^{2n} (\lambda/R)^{-\varepsilon N/8n} \|f\|_{L^2(\hat{\R}^n)}.
\end{equation*}
Choosing $K$ sufficiently large (depending only on $\varepsilon$, $n$, $M$ and $p$) so that $C_{\varepsilon} K^{-\varepsilon/2} \leq 1$, the induction closes and the desired result follows.

%%%%%%%%%%%%%%%%%%%%%%%%%%%%%%%%%%%%%%%%%%%%%%%%%%%%%%%%%%%%%%%%%%%%%%%%%%%%%%%%%%%%%%%%%%%%%%%%

%                                    PROOF OF THE LOCAL SMOOTHING ESTIMATE

%%%%%%%%%%%%%%%%%%%%%%%%%%%%%%%%%%%%%%%%%%%%%%%%%%%%%%%%%%%%%%%%%%%%%%%%%%%%%%%%%%%%%%%%%%%%%%%%

\section{Proof of the local smoothing estimate}\label{section: proof of local smoothing}

In this section the relationships between the theorems stated in the introduction are established and, in particular, it is shown that
\begin{equation*}
\textrm{Theorem~\ref{decoupling theorem}} \Rightarrow \textrm{Theorem~\ref{FIO local smoothing theorem}}  \Rightarrow \textrm{Theorem~\ref{local smoothing theorem}.} 
\end{equation*}
Given the formula for the solution $u$ from \eqref{solution formula}, the latter implication is almost immediate. The former implication follows from a straight-forward adaption of an argument due to Wolff \cite{Wolff2000}, which treats an analogous problem for the euclidean wave equation. Nevertheless, proofs of both of the implications are included for completeness.

To begin, the definition of the cinematic curvature condition, as introduced in \cite{Mockenhaupt1993}, is recalled. As in the statement of Theorem~\ref{FIO local smoothing theorem}, let $Y$ and $Z$ be precompact smooth manifolds of dimensions $n$ and $n+1$, respectively. Let $\mathcal{C} \subseteq T^*Z\setminus 0 \times T^*Y \setminus 0$ be a choice of canonical relation; here $T^*M\setminus 0$ denotes the tangent bundle of a $C^{\infty}$ manifold $M$ with the 0 section removed. Thus, 
\begin{equation*}
\mathcal{C} = \{(x,t,\xi,\tau, y, \eta) : (x,t,\xi,\tau, y, -\eta) \in \Lambda\} 
\end{equation*}
for some conic Lagrangian submanifold $\Lambda \subseteq T^*Z\setminus 0 \times T^*Y \setminus 0$; see \cite{Hormander1971} or \cite{Duistermaat2011, Sogge2017} for further details. Certain conditions are imposed on $\mathcal{C}$, defined in terms of the projections 
\begin{equation*}
\begin{tikzcd}
& \mathcal{C} \arrow[ld, swap, "\Pi_{T^*_{_{}}Y}"]  \arrow[d, "\Pi_{Z}"] \arrow[rd, "\Pi_{T^*_{z_0}Z}"] &   \\
T^*Y\setminus 0 & Z & T^*_{z_0}Z \setminus 0
\end{tikzcd}.
\end{equation*}
First there is the basic non-degeneracy hypothesis that the projections $\Pi_{T^*Y}$ and $\Pi_{Z}$ are submersions: 
\begin{equation}\label{nondegeneracy}
\mathrm{rank}\, \ud \Pi_{T^*Y} \equiv  2n \quad \textrm{and} \quad \mathrm{rank}\, \ud \Pi_{Z} \equiv n+1.
\end{equation}
This condition implies that for each $z_0 \in Z$ the image $\Gamma_{z_0} := \Pi_{T^*_{z_0}Z}(\mathcal{C})$ of $\mathcal{C}$ under the projection onto the fibre $T^*_{z_0}Z$ is a $C^{\infty}$ immersed hypersurface. Note that $\Gamma_{z_0}$ is conic and therefore has everywhere vanishing Gaussian curvature. In addition to the non-degeneracy hypothesis \eqref{nondegeneracy}, the following curvature condition is also assumed:
\begin{equation}\label{cone condition}
\parbox[c][4em][c]{0.8\textwidth}{\centering For all $z_0 \in Z$, the cone $\Gamma_{z_0}$ has $n-1$ non-vanishing principal curvatures at every point.}
\end{equation}
If both \eqref{nondegeneracy} and \eqref{cone condition} hold, then $\mathcal{C}$ is said to satisfy the \emph{cinematic curvature condition} \cite{Mockenhaupt1993}. 

\begin{remark}\label{cinematic curvature remark} Using local coordinates, \eqref{nondegeneracy} and \eqref{cone condition} may be expressed in terms of the conditions H1) and H2) introduced in $\S$\ref{subsection: variable coefficient decoupling}. Indeed, near any point 
\begin{equation*}
(x_0,t_0,\xi_0,\tau_0,y_0,\eta_0) \in \mathcal{C},
\end{equation*}
the condition \eqref{nondegeneracy} implies that there exists a phase function $\phi(z;\eta)$ satisfying H1) such that $\mathcal{C}$ is given locally by
\begin{equation*}
\{ (z,\partial_z \phi(z;\eta), \partial_\eta \phi(z;\eta), \eta ) : \eta \in \R^n\backslash \{ 0\} \:\: \text{in a conic neighbourhood of} \:\: \eta_0\}.
\end{equation*}
Furthermore, \eqref{cone condition} implies that the function $\phi$ satisfies H2). Further details may be found in \cite[Chapter 8]{Sogge2017}.
\end{remark}

Recall from the introduction that the solution to the Cauchy problem \eqref{wave equation} can be written as $u = \mathcal{F}_0f_0 + \mathcal{F}_1f_1$ where each $\mathcal{F}_j \in I^{j -1/4}(M \times \R, M;\mathcal{C})$ for some canonical relation $\mathcal{C}$ satisfying the cinematic curvature condition. Fix a choice of coordinate atlas $\{(\Omega_{\nu}, \kappa_{\nu})\}_{\nu}$ on $M$ and a partition of unity $\{\psi_{\nu}\}_{\nu}$ subordinate to the cover $\{\Omega_{\nu}\}_{\nu}$ of $M$. A choice of Bessel potential norm $\|\,\cdot\,\|_{L^p_s(M)}$ is then defined by
\begin{equation*}
\|f\|_{L^p_s(M)} := \sum_{\nu} \|f_{\nu}\|_{L^p_s(\R^n)}
\end{equation*}
where $f_{\nu} := (\psi_{\nu}f)\circ \kappa^{-1}$ and $L^p_s(\R^n)$ denotes the standard Bessel potential space in $\R^{n}$. Thus, expressing everything in local coordinates and applying the composition theorem for Fourier integral operators (see, for instance, \cite[Theorem 6.2.2]{Sogge2017}), it is clear that Theorem~\ref{local smoothing theorem} is an immediate consequence of Theorem~\ref{FIO local smoothing theorem}.

It remains to show that Theorem~\ref{FIO local smoothing theorem} follows from the decoupling inequality established in Theorem~\ref{decoupling theorem}. Working in local coordinates (and recalling Remark~\ref{cinematic curvature remark} and the discussion in $\S$\ref{subsection: variable coefficient decoupling}), it suffices to prove an estimate for operators of the form
\begin{equation}\label{local smoothing 1}
\mathcal{F}f(x,t) := \int_{\hat{\R}^n} e^{i \phi(x,t; \xi)} b(x,t;\xi)(1 + |\xi|^2)^{\mu/2} \hat{f}(\xi)\,\ud \xi
\end{equation}
where $b$ is a symbol of order 0 with compact support in the $(x,t)$ variables and $\phi$ is a smooth homogenous phase function satisfying H1) and H2) (at least on the support of $b$). Recall that $b$ is a symbol of order 0 if $b \in C^{\infty}(\R^{n+1} \times \R^n)$ and satisfies
\begin{equation*}
|\partial_z^{\nu} \partial_{\xi}^{\gamma} b(z;\xi)| \lesssim_{\gamma, \nu} (1 + |\xi|)^{-|\gamma|} \quad \textrm{for all multi-indices $(\gamma, \nu) \in \N_0^{n+1} \times \N_0^n$.}
\end{equation*}
In particular, Theorem~\ref{FIO local smoothing theorem} is a direct consequence of the following proposition. 
\begin{proposition}\label{FIO estimate}
If $\bar p_n \leq p < \infty$ and $\mathcal{F}$ is defined as in \eqref{local smoothing 1} with $\mu< -\alpha(p)= -\bar{s}_p + 1/p$, then
\begin{equation*}
\| \mathcal{F} f \|_{L^p(\R^{n+1})} \lesssim \| f \|_{L^p(\R^n)}.
\end{equation*}
\end{proposition}

\begin{proof} By applying a rotation and a suitable partition of unity, one may assume that $b$ is supported in $B^n(0,\epsilon_0) \times B^1(1,\epsilon_0) \times \Gamma$ for a suitably small constant $0 < \epsilon_0 \leq 1$ where
\begin{equation*}
\Gamma := \{ \xi \in \hat{\R}^n : |\xi_j| \leq |\xi_n| \textrm{ for }  1 \leq j \leq n-1 \}.
\end{equation*}
Further, as the symbol $b$ has compact $(x,t)$-support of diameter $O(1)$, one may assume without loss of generality that it is of product-type: that is, $b(x,t;\xi) = b_1(x,t)b_2(\xi)$. The latter reduction follows by taking Fourier transforms in a similar manner to that used in the proof of Lemma~\ref{approximation lemma}; the argument, which is standard, is postponed until the end of the proof.  

Fix $\beta \in C^{\infty}_c(\R)$ with $\mathrm{supp}\,\beta \in [1/2,2]$ and such that $\sum_{\lambda \,\mathrm{dyadic}} \beta(r/\lambda) = 1$ for $r \neq 0$. Let $\mathcal{F}^{\lambda} := \mathcal{F} \circ \beta(\sqrt{-\Delta_x}/\lambda)$, so that $\mathcal{F}^{\lambda}f$ is given by introducing a $\beta(|\xi|/\lambda)$ factor into the symbol in \eqref{local smoothing 1},\footnote{In general, $m(i^{-1}\partial_x)$ denotes the Fourier multiplier operator (defined for $f$ belonging to a suitable \emph{a priori} class)
\begin{equation*}
m(i^{-1}\partial_x)f(x) := \int_{\hat{\R}^n} e^{i \langle x, \xi \rangle} m(\xi)  \hat{f}(\xi)\,\ud \xi
\end{equation*}
for any $m \in L^{\infty}(\hat{\R}^n)$. The operator $m(\sqrt{-\Delta_x})$ is then defined in the natural manner via the identity $-\Delta_x = i^{-1}\partial_x \cdot i^{-1}\partial_x$.} and decompose $\mathcal{F}f$ as
\begin{equation*}
\mathcal{F}f =: \mathcal{F}\,^{\lesssim 1}f + \sum_{\lambda \in \N :\, \mathrm{dyadic}} \mathcal{F}^{\lambda}f.
\end{equation*}
It follows that $\mathcal{F}\,^{\lesssim 1}$ is a pseudodifferential operator of order 0 and therefore bounded on $L^p$ for all $1 < p < \infty$. Thus, letting $\varepsilon := -(\mu + \alpha(p))/2 > 0$, the problem is further reduced to showing that
\begin{equation*}
\|\mathcal{F}^{\lambda}f\|_{L^p(\R^{n+1})} \lesssim \lambda^{\alpha(p) + \mu + \varepsilon}\|f\|_{L^p(\R^n)}
\end{equation*}
for all $\lambda \geq 1$.

By various rescaling arguments and Theorem~\ref{decoupling theorem}, it follows that 
\begin{equation*}
\|\mathcal{F}^{\lambda}f\|_{L^p(\R^{n+1})} \lesssim_{s,p} \lambda^{\alpha(p) + \varepsilon} \Big(\sum_{\theta : \lambda^{-1/2}-\mathrm{plate}} \|\mathcal{F}^{\lambda}_{\theta}f\|_{L^p(\R^{n+1})}^p \Big)^{1/p}.
\end{equation*}
where $\mathcal{F}^{\lambda}_{\theta} := \mathcal{F}^{\lambda} \circ a_{\theta}(i^{-1}\partial_x)$ for $a_{\theta}$ a choice of smooth angular cut-off to $\theta$. Thus, to conclude the proof of Proposition~\ref{FIO estimate} (and therefore that of Theorems~\ref{FIO local smoothing theorem} and~\ref{local smoothing theorem}), it suffices to establish the following lemma.

\begin{lemma}\label{recoupling lemma} For $\mathcal{F}^{\lambda}_{\theta}$ as defined above and $2 \leq p \leq \infty$ one has
\begin{equation*}
\Big(\sum_{\theta : \lambda^{-1/2}-\mathrm{plate}} \|\mathcal{F}^{\lambda}_{\theta}f\|_{L^p(\R^{n+1})}^p \Big)^{1/p} \lesssim \lambda^{\mu} \|f\|_{L^p(\R^n)}.
\end{equation*}
\end{lemma}

This inequality essentially appears in \cite{Seeger1991} (see also \cite[Chapter IX]{Stein1993}); a sketch of the proof is included for completeness.

\begin{proof}[Proof (of Lemma~\ref{recoupling lemma})] By interpolation (via H\"older's inequality) it suffices to establish the cases $p = 2$ and $p=\infty$. 

To prove the $p=2$ bound, one may use H\"ormander's theorem (see, for instance, \cite[Chapter IX $\S$1.1]{Stein1993}) for fixed $t$, followed by Plancherel's theorem and the almost orthogonality of the plates $\theta$.

To prove the $p = \infty$ bound, it suffices to show that 
\begin{equation*}
\sup_{(x,t)\in \R^{n+1}}\|K^{\lambda}_{\theta}(x,t;\,\cdot\,)\|_{L^1(\R^n)} \lesssim \lambda^{\mu}
\end{equation*}
where $K^{\lambda}_{\theta}$ is the kernel of the operator $\mathcal{F}^{\lambda}_{\theta}$. This follows from a standard stationary phase argument, which exploits heavily the homogeneity of the phase and the angular localisation; see, for instance, \cite[Chapter IX $\S\S$4.5-4.6]{Stein1993} for further details. 
\end{proof}

It remains to justify the initial reduction to symbols of product-type. As mentioned earlier, the argument is standard and appears, for instance, in the proof of the $L^2$ boundedness for pseudodifferential operators of order 0 (see \cite[Chapter VI, $\S$2]{Stein1993}). 

As $b$ is a symbol of order 0 with compact $(x,t)$-support, $(n+2)$-fold integration-by-parts shows that 
\begin{equation}\label{decay FT symbol}
|\partial_{\xi}^{\gamma}\hat{b}(\zeta;\xi)| \lesssim_{\gamma} (1 + |\zeta|)^{-(n+2)}(1+|\xi|)^{-|\gamma|} \qquad \textrm{for all multi-indices $\gamma \in \N_0^n$,}
\end{equation}
where $\hat{b}$ denotes the Fourier transform of $b$ in the $z = (x,t)$ variable. By means of the Fourier transform one may write 
\begin{equation*}
\mathcal{F} f (x,t) = \int_{\hat{\R}^{n+1}} e^{i  \langle z, \zeta \rangle } (1+ |\zeta|)^{-(n+2)}  \int_{\hat{\R}^n} e^{i \phi(x,t;\xi)} \frac{b_{\zeta}(x,t;\xi)}{(1+|\xi|^2)^{-\mu/2}}\widehat{f}(\xi) \ud\xi  \ud\zeta,
\end{equation*}
where $b_{\zeta}(x,t;\xi) := \psi(x,t)\hat{b}(\zeta; \xi) (1+ |\zeta|)^{n+2}$ for $\psi$ a smooth cut-off equal to 1 in the $z$-support of $b$ and vanishing outside its double. The functions $b_{\zeta}$ are all of product-type and, by \eqref{decay FT symbol}, are symbols of order 0 uniformly in $\zeta$. Taking $L^p$-norms and applying Minkowski's integral inequality, it now suffices to show the $L^p$ boundedness of $\mathcal{F}$ under the product hypothesis.
\end{proof}

%%%%%%%%%%%%%%%%%%%%%%%%%%%%%%%%%%%%%%%%%%%%%%%%%%%%%%%%%%%%%%%%%%%%%%%%%%%%%%%%%%%%%%%%%%%%%%%%

%                               COUNTEREXAMPLES FOR LOCAL SMOOTHING ESTIMATES FOR CERTAIN FOURIER INTEGRAL OPERATORS

%%%%%%%%%%%%%%%%%%%%%%%%%%%%%%%%%%%%%%%%%%%%%%%%%%%%%%%%%%%%%%%%%%%%%%%%%%%%%%%%%%%%%%%%%%%%%%%%

\section{Counterexamples for local smoothing estimates for certain Fourier integral operators}\label{section: sharp examples}

To conclude the paper the proof of Proposition~\ref{sharpness proposition} is presented. As originally observed by the third author in \cite{Sogge1991} and elaborated further in, for instance, \cite{Mattila2015, Mockenhaupt1993, Sogge2017, Tao1999}, it is known that local smoothing estimates for Fourier integral operators imply favourable $L^p$ estimates for a natural class of oscillatory integral operators. Indeed, this is the basis of the well-known formal implication that the local smoothing conjecture for the (euclidean) wave equation implies the Bochner--Riesz conjecture (see \cite{Sogge1991} or \cite{Sogge2017}). In this section a general form of this implication is combined with a counterexample of Bourgain \cite{Bourgain1991, Bourgain1995} to show that Theorem~\ref{FIO local smoothing theorem} is sharp when $n \geq 3$ is odd.

%%%%%%%%%%%%%%%%%%%%%%%%%%%%%%%%%%%%%%%%%%%%%%%%%%%%%%%%%%%%%%%%%%%%%%%%%%%%%%%%%%%%%%%%%%%%%%%%

%                                          LOCAL SMOOTHING FOR FOURIER INTEGRALS AND NON-HOMOGENEOUS OSCILLATORY INTEGRALS

%%%%%%%%%%%%%%%%%%%%%%%%%%%%%%%%%%%%%%%%%%%%%%%%%%%%%%%%%%%%%%%%%%%%%%%%%%%%%%%%%%%%%%%%%%%%%%%%

\subsection{Local smoothing for Fourier integrals and non-homogeneous oscillatory integrals}

Let $\Omega \subseteq \R^n$ be open and suppose that $\Phi \colon \Omega \times \Omega \to \R$ is smooth and satisfies  
\begin{equation}\label{non-singular Phi}
\partial_{y}\Phi(x,y) \neq 0 \qquad \textrm{for all $(x,y) \in \Omega \times \Omega$}
\end{equation}
and, moreover, that the Monge--Ampere matrix associated to $\Phi$ is everywhere non-singular:
\begin{equation}\label{MA}
\det \left(\begin{array}{cc}
0                         &    \partial_y\Phi(x,y)            \\[4pt]
\partial_x\Phi(x,y) &     \partial^2_{xy}\Phi(x,y)
\end{array}\right)\neq 0 
\qquad \textrm{for all $(x,y) \in \Omega \times \Omega$.}
\end{equation}
By \eqref{non-singular Phi}, for each $(x,t) \in \Omega \times (-1,1)$ the level set
\begin{equation}\label{rotcurv}
S_{x,t} := \big\{y\in \Omega:\Phi(x,y)=t \big\}
\end{equation}
is a smooth hypersurface. The condition \eqref{MA} implies that the smooth family of surfaces in \eqref{rotcurv} satisfies the \emph{rotational curvature condition} of Phong and Stein \cite{Phong1986} (see also \cite[Chapter XI]{Stein1993}). 

The above phase function can be used to construct two natural oscillatory integral operators. To describe these objects, first fix a pair of smooth cut-off functions $a\in C^\infty_c(\Omega\times \Omega)$ and $\rho\in C^\infty_c((-1,1))$. 

\begin{enumerate}[i)]
\item For each fixed $t \in \R$ the distribution
\begin{equation}\label{FIO kernel}
K(x,t;y) := \rho(t)\,  a(x,y) \,  \delta_0(t-\Phi(x,y))
\end{equation}
is the kernel of a conormal Fourier integral operator on $\R^n \times \R^n$ of order $-(n-1)/2$. In particular, $K$ can be written as
\begin{equation*}
K(x,t;y)  = \int_{\hat{\R}} e^{i \tau(t - \Phi(x,y))}\rho(t)  a(x,y) \,  \ud \tau,
\end{equation*}
where the right-hand side expression is understood to converge in the sense of oscillatory integrals. From this formula, one can easily deduce (using, for instance, \cite[Theorem 0.5.1]{Sogge2017}) that the canonical relation is given by
\begin{equation}\label{canonical averaging}
\mathcal{C}=\{(x,t, - \tau \partial_x \Phi(x,y), \tau, y, \tau\partial_y \Phi(x,y) ) : \:  \Phi(x,y)=t  )\}.
\end{equation}
It is remarked that the condition \eqref{MA} ensures that each of these Fourier integrals is non-degenerate in the sense that the canonical relation is a canonical graph. 

It will be useful to consider the operator
 \begin{equation}\label{F}
\mathcal{F}f(x,t) := \int_{\R^n} K(x,t;y) f(y) \ud y,
 \end{equation}
which is understood to map functions on $\R^n$ to functions on $\R^{n+1}$ by taking averages over the variable hypersurfaces $S_{x,t}$.
\item One may also consider the non-homogeneous oscillatory integral operator
 \begin{equation}\label{NH} 
 S_{\Phi}^{\lambda} f(x) :=\int_{\R^n} e^{i\lambda \Phi(x,y)} a(x,y)  f(y) \ud y,
 \end{equation}
 where the amplitude $a\in C^\infty_c(\Omega \times \Omega)$ is as in \eqref{FIO kernel} and $\lambda \geq 1$.
\end{enumerate}

Assume, in addition to the condition \eqref{MA}, that
 \begin{equation}\label{del}
 \rho(t)\delta_0(t- \Phi(x,y))=\delta_0(t-\Phi(x,y)) \quad \textrm{for all $(x,y) \in \mathrm{supp}\, a$ and $t \in \R$.}
 \end{equation}
Note that this holds if, for instance, $\Phi(0,0)=0$ and $\rho(t)=1$ for all $t$ in a neighbourhood of 0 provided that $a$ vanishes outside of a small neighbourhood of the origin in $\R^n\times\R^n$. Under these conditions $L^p$ bounds for the operator \eqref{NH} can be related to Sobolev estimates for \eqref{F}.

\begin{proposition}\label{Fourier vrs oscillatory proposition} Under the conditions \eqref{MA} and \eqref{del}, if $\gamma>0$ is fixed and $\lambda \geq 1$, then
\begin{equation}\label{key}
\| S_{\Phi}^{\lambda} \|_{L^p(\R^n)\to L^p(\R^n)} \lesssim \lambda^{-\gamma} \bigl\| \, (I-\Delta_x)^{\gamma/2} \circ {\mathcal F} \, \bigr\|_{L^p(\R^n)\to L^p(\R^{n+1})}.
\end{equation}
\end{proposition}

\begin{proof} Let $\beta\in C_c^{\infty}(\R)$ satisfy $\beta(r) = 1$ for $|r| \leq 1$ and $\beta(r)=0$ for $|r| \geq 2$. The condition \eqref{MA} implies that $\partial_x \Phi(x,y)\ne 0$ for all $(x,y) \in \mathrm{supp}\,a$ and a simple integration-by-parts argument therefore shows that for some small constant $c_0>0$ the estimate
\begin{equation*}
\bigl\| \, \beta\bigl(\sqrt{-\Delta_x}/c_0\lambda\bigr)  \circ S_{\Phi}^{\lambda} \, \bigr\|_{L^p(\R^n)\to L^p(\R^n)}=O_N(\lambda^{-N})
\end{equation*}
holds for all $N \in \N$. Furthermore, since $\gamma>0$, it also follows that
\begin{equation*}
\bigl\|(1-\beta(\sqrt{-\Delta_x}/c_0\lambda)) \circ (I-\Delta_x)^{-\gamma/2} \bigr\|_{L^p(\R^n)\to L^p(\R^n)}=O(\lambda^{-\gamma}).
\end{equation*}
Combining these observations, 
\begin{equation}\label{Fourier vrs oscillatory proposition 1} 
\|S_{\Phi}^{\lambda}\|_{L^p(\R^n)\to L^p(\R^n)}\lesssim \lambda^{-\gamma} \bigl\|\, (I-\Delta_x)^{\gamma/2}\circ S_{\Phi}^{\lambda}\, \bigr\|_{L^p(\R^n)\to L^p(\R^n)} + O_N(\lambda^{-N}).
\end{equation}
On the other hand, the definition of $K$ and the condition \eqref{del} imply that
\begin{equation*}
\int e^{i\lambda t}\, K(x,t;y) \,\ud t = e^{i\lambda \Phi(x,y)}a(x,y).
\end{equation*}
One may therefore write the operator $S_{\Phi}^{\lambda}$ in terms of $K$ and apply H\"older's inequality together with the estimate \eqref{Fourier vrs oscillatory proposition 1} to deduce the desired result. 
\end{proof}

%%%%%%%%%%%%%%%%%%%%%%%%%%%%%%%%%%%%%%%%%%%%%%%%%%%%%%%%%%%%%%%%%%%%%%%%%%%%%%%%%%%%%%%%%%%%%%%%

%                                          SHARPNESS OF THE RANGE OF EXPONENTS p FOR THE OPTIMAL LOCAL SMOOTHING BOUNDS FOR ODD n

%%%%%%%%%%%%%%%%%%%%%%%%%%%%%%%%%%%%%%%%%%%%%%%%%%%%%%%%%%%%%%%%%%%%%%%%%%%%%%%%%%%%%%%%%%%%%%%%

\subsection{Sharpness of the range of exponents $p\geq \bar{p}_n $ for optimal local smoothing bounds for odd $n$}

To show that the bounds obtained in Theorem~\ref{FIO local smoothing theorem} are sharp in odd dimensions, in this section certain phase functions $\Phi$ are investigated which, in addition to \eqref{MA}, satisfy a variant of the Carleson--Sj\"olin condition from \cite{Carleson1972}.

Note that \eqref{MA} ensures that at each point the rank of the mixed Hessian of $\Phi$ is at least $n-1$.  Assume that
\begin{equation}\label{cs1}
\mathrm{rank}\, \partial_{xy}^2 \Phi(x,y) = n-1 \quad \textrm{for all $(x,y) \in \mathrm{supp} \, a$.}
\end{equation}
It then follows that, provided $\Omega$ is sufficiently small, for any fixed $x_0$ in the $x$-support of $a$ the map
\begin{equation*}
y\to \partial_x \Phi(x_0,y), \qquad y\in \Omega
\end{equation*}
parametrises a hypersurface $ \Sigma_{x_0} \subset \R^n$. Suppose, in addition to \eqref{cs1}, the phase also satisfies the following curvature condition:
\begin{equation}\label{curv}
\parbox[c][4em][c]{0.8\textwidth}{\centering For each $x_0\in \Omega$ the surface $\Sigma_{x_0}$ has $n-1$ non-vanishing principal curvatures at every point.}
\end{equation}
In this case, the phase function $\Phi$ is said to satisfy the \emph{$n\times n$ Carleson--Sj\"olin condition} (see \cite{Sogge2017}). This definition should be compared with the similar conditions
H1) and H2) for the homogeneous oscillatory integrals described in $\S$\ref{subsection: variable coefficient decoupling} (note, for instance, that \eqref{curv} is equivalent to the condition that, for a suitably defined Gauss map $G_{\Phi}$, the $y$-Hessian of $\langle \partial_x\Phi(x_0,y),G_{\Phi}(x_0,y_0)\rangle$ has rank $n-1$ at $y = y_0$ for every $(x_0,y_0) \in \Omega$). 

If \eqref{cs1} and \eqref{curv} are valid, then it is claimed that the Fourier integral operators ${\mathcal F}$ in \eqref{F} satisfy the cinematic curvature condition appearing in the hypotheses of Theorem~\ref{FIO local smoothing theorem}.  If ${\mathcal C}\subset T^*\R^n \,\setminus\, 0 \times  T^*\R^{n+1} \, \setminus \, 0$ is the canonical relation for ${\mathcal F}$, then recall that the non-degeneracy condition \eqref{nondegeneracy} is that $\textrm{rank}\,\ud\Pi_{T^*\R^n}\equiv 2n$ and $\textrm{rank}\,\ud\Pi_{\R^{n+1}}\equiv n+1$. This holds as an immediate consequence of \eqref{MA} since, as was observed earlier, \eqref{MA} implies that ${\mathcal C}$ is a local canonical graph. It remains to verify the cone condition \eqref{cone condition}. It immediately follows from the expression \eqref{canonical averaging} that for the Fourier integral operators in \eqref{F} the cones $\Gamma_{x_0,t_0}$ are given by
\begin{equation*}
\Gamma_{x_0,t_0}=\bigl\{ \tau(-\partial_x\Phi(x_0,y), 1): y\in \Omega,  \tau \in \R \bigr\}.
\end{equation*}
Consequently, the cone condition holds if \eqref{cs1} and \eqref{curv} are satisfied. This verifies the claim.

Recall from the discussion following the statement of Proposition~\ref{sharpness proposition} that for each fixed $t$ the composition
\begin{equation*}
(I-\Delta_x)^{\gamma/2} \circ \bigl({\mathcal F}h\bigr)(\, \cdot \, ,t)
\end{equation*}
is a Fourier integral operator of order $-(n-1)/2+\gamma$. Thus, a special case of the local smoothing problem is to show that for a given exponent $2n/(n-1) \leq p < \infty$ one has
\begin{equation}\label{ls}
\bigl\| \, (I-\Delta_x)^{\gamma/2} \circ {\mathcal F}\, \bigr\|_{L^p(\R^n)\to L^p(\R^{n+1})}=O(1) \quad \textrm{for all $0<\gamma<n/p$}.
\end{equation}
Note that, unlike the operators in \eqref{F}, the Fourier integrals in \eqref{ls} do not have kernels with compact $x$-support; however, they are bounded and rapidly decreasing outside of any neighbourhood of the $x$-support of $a$.

Adapting a counterexample of Bourgain \cite{Bourgain1991, Bourgain1995}, one may construct a phase $\Phi$ so that \eqref{ls} cannot hold for $p < \bar{p}_n$ if $n\ge 3$ is odd. This establishes Proposition~\ref{sharpness proposition} and thereby shows that Theorem~\ref{FIO local smoothing theorem} is optimal in the odd-dimensional case. The details are given presently. It is remarked that, strictly speaking, here a slight simplification of Bourgain's construction is used, which is due to Stein \cite[Chapter IX, $\S$6.5]{Stein1993} (see also \cite[pp. 67-69]{Sogge2017} for further details). 

\begin{proof}[Proof (of Proposition~\ref{sharpness proposition})] Consider the matrix-valued function $A \colon \R \to \mathrm{Mat}(2, \R)$ defined by
\begin{equation*}
A(s):=
\begin{pmatrix}
1 &s
\\
s &s^2
\end{pmatrix} \qquad \textrm{for all $s \in \R$.}
\end{equation*}
Let $n\geq 3$ be odd and $\bm{A} \colon \R \to \mathrm{Mat}(n-1, \R)$ be given by
\begin{equation*}
\bm{A}(s) := \underbrace{A(s) \oplus \dots \oplus A(s)}_{\textrm{$ \frac{n-1}{2} $-fold}}
\end{equation*}
so that $\bm{A}(s)$ is an $(n-1)\times(n-1)$ block-diagonal matrix. Using these matrices, define a phase function $\phi$ on $\R^n\times \R^{n-1}$ by
\begin{equation}\label{odd}
\phi(x,y'): = \langle x', y' \rangle +\frac{1}{2} \langle \bm{A}(x_n)y', y'\rangle
\end{equation}
for all $(x,y') = (x',x_n,y') \in \R^n \times \R^{n-1}$. Given an amplitude function $b \in C^{\infty}_c(\R^n \times \R^{n-1})$ define the oscillatory integral operator
\begin{equation*}
S_{\phi}^{\lambda} f(x) := \int_{\R^{n-1}} e^{i\lambda \phi(x,y')} \, b(x,y') f(y') \, \ud y'.
\end{equation*}
A stationary phase argument (see, for instance, \cite[pp. 68-69]{Sogge2017}) then yields
\begin{equation}\label{bad p}
\lambda^{-(n-1)/4 -(n-1)/2p} \lesssim \|S_{\phi}^{\lambda}\|_{L^p(\R^{n-1})\to L^p(\R^n)}, \quad
\textrm{if $\lambda \geq 1$ and $p\ge 2$},
\end{equation}
provided that $b(0,0)\ne 0$.

If $\phi$ is as in \eqref{odd} and
\begin{equation}\label{bad Phi}
\Phi(x,y) := \phi(x,y')+x_n+y_n,
\end{equation}
then clearly \eqref{MA} is valid when $x=y=0$.  Since
\begin{equation*}
y \to \partial_x \Phi(0,y)=(y', \sum_{j=0}^{(n-3)/2} y_{2j+1}y_{2j+2})+e_n
\end{equation*}
parametrises a hyperbolic paraboloid with $(n-1)/2$ positive principal curvatures and $(n-1)/2$ negative principal curvatures, one concludes that for small $x$ the Carleson--Sj\"olin conditions \eqref{cs1} and \eqref{curv} must hold, provided the support of $b$ lies in a suitably small ball about the origin. 

Suppose $\beta$ is as in the proof of Proposition~\ref{Fourier vrs oscillatory proposition}, so that $\beta \in C^{\infty}_c(\R)$ satisfies $\beta(r) = 1$ whenever $|r| \leq 1$ and $\beta(r) = 0$ whenever $|r| \geq 2$. Assume $b \in C^{\infty}_c(\R^n \times \R^{n-1})$ satisfies  $b(0,0)\ne 0$ and is supported in a small neighbourhood of the origin. Take $a$ in \eqref{NH} to be equal to
\begin{equation*}
a(x,y)=b(x,y')\beta(y_n/ c_0)
\end{equation*}
for some suitable choice of small constant $0 < c_0 < 1/2$. Provided the size of the support of $b$ and $c_0$ are suitably chosen, \eqref{del} holds. Furthermore, taking $F(y) := \beta(y_n)e^{-i\lambda y_n}f(y')$  in \eqref{NH}, one readily observes that
\begin{equation*}
|S_{\phi}^{\lambda} f(x)| \sim |S_{\Phi}^{\lambda} F(x)| \, \, \text{and}
\quad  \, \, \|F\|_{L^p(\R^n)} \sim \|f\|_{L^p(\R^{n-1})} 
\end{equation*}
and, consequently, 
\begin{equation*}
\|S_{\phi}^{\lambda} \|_{L^p(\R^{n-1}) \to L^p(\R^n)} \lesssim \|S_{\Phi}^{\lambda}\|_{L^p(\R^n)\to L^p(\R^n)}.
\end{equation*}
Combining this with \eqref{bad p} and \eqref{key}, for $\gamma>0$ and $\lambda \geq 1$ one concludes that
\begin{equation*}
\lambda^{\gamma -(n-1)/4-(n-1)/2p} \lesssim \bigl\| \, (I-\Delta_x)^{\gamma/2} \circ {\mathcal F} \, \bigr\|_{L^p(\R^n)\to L^p(\R^{n+1})}
\end{equation*}
where ${\mathcal F}$ is as in \eqref{F}.  Since
\begin{equation*}
\frac{n}{p} -\frac{n-1}{4}-\frac{n-1}{2p}>0 \quad \mathrm{if} \quad p< \bar{p}_n,
\end{equation*}
it follows that \eqref{ls} cannot hold for any Lebesgue exponent satisfying $p<\bar{p}_n$.
\end{proof}

For even dimensions $n\ge 4$ one may modify the argument given in the proof of Proposition~\ref{sharpness proposition} to give a necessary condition for the local smoothing problem for the general class of Fourier integral operators under consideration. Indeed, in the even dimensional case one simply defines 
 \begin{equation*}
\bm{A}(s) := \underbrace{A(s)\oplus \dots \oplus A(s)}_{\textrm{$\frac{n-2}{2}$-fold}} \oplus (1+s),
\end{equation*}
where $(1+s)$ is a $1 \times 1$ matrix with entry $1+s$, so that once again $\bm{A}(s)$ is an $(n-1)\times(n-1)$ block-diagonal matrix. Taking the phase function $\phi$ as in \eqref{odd}, it follows that the resulting oscillatory integral operators satisfy
\begin{equation*}
\lambda^{-n/4-(n-2)/2p}\lesssim \|S_{\phi}^{\lambda}\|_{L^\infty(\R^{n-1})\to L^p(\R^n)}.
\end{equation*}
See, for instance, \cite[p. 69]{Sogge2017} for further details. Arguing \emph{mutatis mutandis}, for even $n\ge 4$ and $\mathcal{F}$ defined as in the proof of Proposition~\ref{sharpness proposition} (with respect to the new choice of phase $\phi$) the estimate \eqref{ls} fails for $p<2(n+2)/n$.

%%%%%%%%%%%%%%%%%%%%%%%%%%%%%%%%%%%%%%%%%%%%%%%%%%%%%%%%%%%%%%%%%%%%%%%%%%%%%%%%%%%%%%%%%%%%%%%%

%                                                    SOME OPEN PROBLEMS

%%%%%%%%%%%%%%%%%%%%%%%%%%%%%%%%%%%%%%%%%%%%%%%%%%%%%%%%%%%%%%%%%%%%%%%%%%%%%%%%%%%%%%%%%%%%%%%%

\subsection{Some open problems} The cones $\Gamma_{x_0,t_0}\subset T^*_{x_0,t_0}\R^{n+1}$ associated to the phase in \eqref{bad Phi} have principal curvatures of opposite sign (in fact, in the examples considered above the difference between the number of positive and the number of negative principal curvatures is minimal).  It would be interesting to see if any improvement is possible in the range of $p$ for which there is optimal local smoothing if the $\Gamma_{x_0,t_0}$ always have $n-1$ positive principal curvatures. The model case for this is the class of Fourier integrals arising in the context of Theorem~\ref{local smoothing theorem}: that is, from solutions of wave equations given by a Laplace--Beltrami operator on some Riemannian manifold $(M,g)$. In this case $\Phi(x,y)$ is given by the associated Riemannian distance function $d_g(x,y)$ minus a constant.  By Proposition~\ref{Fourier vrs oscillatory proposition} and the counterexamples of Minicozzi and the third author \cite{Minicozzi1997} (see also \cite{Sogge}), there exist metrics for which optimal local smoothing is not possible when $p < \bar{p}_{n, +}$ where
\begin{equation*}
\bar{p}_{n, +} := \left\{
\begin{array}{ll}
\tfrac{2(3n+1)}{3n-3} & \textrm{if $n$ is odd}, \\[4pt]
\tfrac{2(3n+2)}{3n-2} & \textrm{if $n$ is even.}
\end{array} \right. 
\end{equation*}

On the other hand, if $\Phi(x,y) := d_g(x,y)$, then recent results of Guth, Iliopoulou and the second author \cite{Guth} yield the optimal bounds for $p \geq \bar{p}_{n, +}$ for the oscillatory operators in \eqref{NH}; this suggests that one should be able to obtain optimal local smoothing bounds for $p \geq \bar{p}_{n, +}$ under the above convexity assumptions. In Figure~\ref{exponents table} the conjectured numerology for sharp local smoothing estimates for Fourier integral operators is tabulated, according to the parity of the dimension and various curvature assumptions. As mentioned in the introduction, for the euclidean wave equation sharp local smoothing estimates are conjectured to hold for the wider range $2n/(n-1)\leq p<\infty$.

\begin{figure}
\begin{center}
\begin{TAB}(c,1cm,2cm)[5pt]{|c|c|c|}{|c|c|c|}
 & $n$ odd & $n$ even    \\
\begin{tabular}[x]{@{}c@{}}$n-1$ non-vanishing \\curvatures\end{tabular} & $\displaystyle \frac{2(n+1)}{n-1}$ & $\displaystyle \frac{2(n+2)}{n}$  \\
\begin{tabular}[x]{@{}c@{}}$n-1$ positive \\curvatures\end{tabular}  & $\displaystyle \frac{2(3n+1)}{3n-3}$ & $\displaystyle \frac{2(3n+2)}{3n-2}$ \\
\end{TAB}
\caption{Conjectured endpoint values for the exponent $p$ for the sharp local smoothing estimates in Theorem~\ref{FIO local smoothing theorem} under various hypothesis on $\mathcal{F} \in I^{\mu-1/4}$. Theorem~\ref{FIO local smoothing theorem} establishes the odd dimensional case under the hypothesis of $n-1$ non-vanishing principal curvatures.}
\label{exponents table}
\end{center}
\end{figure}

Finally, it is remarked that the conjectured numerology in Figure~\ref{exponents table} coincides with the sharp bounds to a problem posed by H\"ormander \cite{Hormander1973} for oscillatory integral operators of the type $T^\lambda$ under non-homogeneous versions of the conditions H1) and H2) (and a corresponding positive-definite version of H2)); see \cite{Guth} for the details of this problem and a full historical account. In particular, the argument presented earlier in this section shows that Theorem~\ref{local smoothing theorem} implies a theorem of Stein \cite{Stein1986} in this context. For the remaining cases, the results of Bourgain \cite{Bourgain1991, Bourgain1995}, Wisewell \cite{Wisewell2005}, Bourgain--Guth \cite{Bourgain2011} and Guth, Iliopoulou and the second author \cite{Guth} suggest the $p \geq 2(n+2)/n$ numerology in the general even dimensional case and reinforce the conjectured $p \geq \bar{p}_{n, +}$ numerology in the convex case.

%%%%%%%%%%%%%%%%%%%%%%%%%%%%%%%%%%%%%%%%%%%%%%%%%%%%%%%%%%%%%%%%%%%%%%%%%%%%%%%%%%%%%%%%%%%%%%%%

%                                          REFERENCES

%%%%%%%%%%%%%%%%%%%%%%%%%%%%%%%%%%%%%%%%%%%%%%%%%%%%%%%%%%%%%%%%%%%%%%%%%%%%%%%%%%%%%%%%%%%%%%%%

\bibliography{Reference}
\bibliographystyle{amsplain}

\end{document}